\pdfoutput=1
\documentclass[onefignum,onetabnum]{siamart190516}
\PassOptionsToPackage{square,numbers}{natbib}
\usepackage[utf8]{inputenc} 
\usepackage[T1]{fontenc}    
\usepackage{url}            
\usepackage{booktabs}       
\usepackage{amsfonts}       
\usepackage{nicefrac}       
\usepackage{microtype}      
\usepackage{mathtools}
\usepackage{amsfonts}
\usepackage{dsfont}
\usepackage{mathrsfs}
\usepackage{tikz}
\usetikzlibrary{calc}
\usepackage{relsize}
\usepackage{scalerel}
\usepackage{comment}
\tikzset{fontscale/.style = {font=\relsize{#1}}}
\usepackage{pgfplots}
\usepackage{color}
\usepackage{graphicx}
\graphicspath{{Figs/}}
\usepackage{indentfirst,latexsym,bm} 
\usepackage{longtable}
\usepackage{cuted}
\usepackage{booktabs,multirow,array,multicol}
\usepackage{enumitem}

\hypersetup{hypertex=true,
            colorlinks=true,
            linkcolor=blue,
            anchorcolor=blue,
            citecolor=blue}
             
\usepackage{subfigure}
\usepackage{amssymb}
\definecolor{ivory}{RGB}{218,215,203}

\definecolor{cuhkp}{RGB}{98,56,105} 	
\definecolor{cuhkpl}{RGB}{152,24,147} 	
\definecolor{cuhkb}{RGB}{219,160,1} 	
\definecolor{cuhkbd}{RGB}{178,129,0} 	
\definecolor{cuhkr}{RGB}{88,35,155}  	

\definecolor{oranl}{RGB}{223,149,86}
\definecolor{turql}{RGB}{53,130,134}  		

\definecolor{blackp}{RGB}{0,0,0} 
\definecolor{redp}{RGB}{255,0,0}
\definecolor{orangep}{RGB}{255,128,0}
\definecolor{brownp}{RGB}{128,77,0}
\definecolor{yellowp}{RGB}{255,230,0}
\definecolor{greenp}{RGB}{128,230,0}
\definecolor{bluep}{RGB}{0,128,255}
\definecolor{purplep}{RGB}{152,24,147}
\definecolor{pinkp}{RGB}{230,0,128}       
\usepackage{color}
\definecolor{amg}{RGB}{17,140,17}
\newtheorem{thm}{theorem}[section]

\newtheorem{fact}[thm]{Fact}
\newtheorem{lem}[thm]{Lemma}
\newtheorem{prop}[thm]{Proposition}

\newtheorem{defn}[thm]{Definition}

\newtheorem{example}[thm]{Example}

\DeclareMathOperator*{\argmin}{argmin}
 \newcommand\conv{\mathrm{conv}}

\newcommand{\RNum}[1]{\uppercase\expandafter{\romannumeral #1\relax}}
\newcommand{\rnum}[1]{\expandafter{\romannumeral #1\relax}}

\newcommand\cM{\mathcal{M}}
\newcommand\cA{\mathcal{A}}
\newcommand\cC{\mathcal{C}}

\newcommand\cR{\mathcal{R}}

\newcommand\cS{\mathcal{S}}
\newcommand\cH{\mathcal{H}}
\newcommand\parr{\mathrm{par}}
\newcommand\epii{\mathrm{epi}}
\newcommand\epi{\mathrm{epi}}

\newcommand\cN{\mathcal{N}}
\newcommand\tr{\mathrm{tr}}
\newcommand\cl{\mathrm{cl}}
\newcommand\cL{\mathcal{L}}

\newcommand\cB{\mathcal{B}}
\newcommand\veps{\varepsilon}
\newcommand\xopt{\bar x}

\newcommand\cK{\mathcal{K}}
\newcommand\bS{\mathbb{S}}
\newcommand\ri{\mathrm{ri}}
\newcommand\rbd{\mathrm{rbd}}
\newcommand\dist{\mathrm{dist}}
\newcommand\dom{\mathrm{dom}}
\newcommand\lin{\mathrm{lin}}

\newcommand\rd{\mathrm{d}}

\newcommand{\elimit}{\xrightarrow{e}}

\newcommand\aff{\mathrm{aff}}
\newcommand\B{\mathbb{B}}

\def\d{\mathrm{d}}
\newcommand\proxs{\mathrm{prox}_{\tau\varphi}}

\newcommand\R{\mathbb{R}}
\newcommand\Rn{\mathbb{R}^n}

\newcommand\Rm{\mathbb{R}^m}
\newcommand\n{\mathbb{N}}
\newcommand\Rex{(-\infty,\infty]}
\newcommand\Rexx{[-\infty,\infty]}

\newcommand{\iprod}[2]{\langle #1, #2 \rangle}
\newcommand\vp{\varphi}
\newcommand\half{\frac12}

\newcommand\gph{\mathrm{gph}}
\newcommand\envs{\mathrm{env}_{\tau\varphi}}

\newcommand{\be}{\begin{equation}}
\newcommand{\ee}{\end{equation}}

\newcommand{\comnew}[1]{\textcolor{turql}{#1}}

\makeatletter

\newcommand{\Rmnum}[1]{\expandafter\@slowromancap\romannumeral #1@}
\makeatother

\newcommand\pvpd{\mathcal{Q}}
\newcommand\usotp{\mathcal{UTP}}

\newcommand\propN{\mathcal{N}}

\newcommand{\tq}[1]{\textcolor{turql}{#1}}

\RequirePackage[capitalize,nameinlink]{cleveref}[0.19]

\crefname{section}{section}{sections}
\crefname{subsection}{subsection}{subsections}

\Crefname{figure}{Figure}{Figures}

\crefformat{equation}{\textup{#2(#1)#3}}
\crefrangeformat{equation}{\textup{#3(#1)#4--#5(#2)#6}}
\crefmultiformat{equation}{\textup{#2(#1)#3}}{ and \textup{#2(#1)#3}}
{, \textup{#2(#1)#3}}{, and \textup{#2(#1)#3}}
\crefrangemultiformat{equation}{\textup{#3(#1)#4--#5(#2)#6}}%
{ and \textup{#3(#1)#4--#5(#2)#6}}{, \textup{#3(#1)#4--#5(#2)#6}}{, and \textup{#3(#1)#4--#5(#2)#6}}

\Crefformat{equation}{#2Equation~\textup{(#1)}#3}
\Crefrangeformat{equation}{Equations~\textup{#3(#1)#4--#5(#2)#6}}
\Crefmultiformat{equation}{Equations~\textup{#2(#1)#3}}{ and \textup{#2(#1)#3}}
{, \textup{#2(#1)#3}}{, and \textup{#2(#1)#3}}
\Crefrangemultiformat{equation}{Equations~\textup{#3(#1)#4--#5(#2)#6}}%
{ and \textup{#3(#1)#4--#5(#2)#6}}{, \textup{#3(#1)#4--#5(#2)#6}}{, and \textup{#3(#1)#4--#5(#2)#6}}

\crefdefaultlabelformat{#2\textup{#1}#3}

\usepackage{lipsum}
\usepackage{amsfonts}
\usepackage{graphicx}
\usepackage{epstopdf}
\usepackage{algorithmic}
\ifpdf
  \DeclareGraphicsExtensions{.eps,.pdf,.png,.jpg}
\else
  \DeclareGraphicsExtensions{.eps}
\fi

\newsiamremark{remark}{Remark}
\newsiamremark{hypothesis}{Hypothesis}
\crefname{hypothesis}{Hypothesis}{Hypotheses}
\newsiamthm{claim}{Claim}
\newsiamthm{assumption}{Assumption}
\headers{Decomposable Functions: Strict Epi-Calculus and Applications}{W. Ouyang and A. Milzarek}

\title{Variational Properties of Decomposable Functions. Part I: Strict Epi-Calculus and Applications\thanks{Submitted to the editors DATE.
\funding{Andre Milzarek was partly supported 
by the Shenzhen Science and Technology Program under Grant No. RCYX20221008093033010 and by the Guangdong Key Lab of Mathematical Foundations for Artificial Intelligence (2023B1212010001).}}}

\author{Wenqing Ouyang\thanks{School of Data Science, The Chinese University of Hong Kong, Shenzhen (CUHK-Shenzhen), Guangdong, 518172, P.R. China (\email{wenqingouyang1@link.cuhk.edu.cn} and \email{andremilzarek@cuhk.edu.cn}).}
  \and Andre Milzarek\footnotemark[2]
 }

\usepackage{amsopn}
\DeclareMathOperator{\diag}{diag}

\hypersetup{
  pdftitle={Variational Properties of Decomposable Functions. Part II: Strong Second-Order Theory},
  pdfauthor={W. Ouyang and A. Milzarek}
}

\begin{document}

\maketitle

\begin{abstract}
This work provides a systematic study of the variational properties of decomposable functions which are compositions of an outer support function and an inner smooth mapping under certain constraint qualifications. A particular focus is put on the strict twice epi-differentiability and the associated strict second subderivative of such functions. A lower bound for the strict second subderivative of decomposable functions is derived which allows linking the strict second subderivative of decomposable mappings to the simpler outer support function. Leveraging the variational properties of the support function, we establish the equivalence between the strict twice epi-differentiability of decomposable functions, continuous differentiability of the proximity operator, and the strict complementarity condition. As an application, this allows us to fully characterize the strict saddle point property of decomposable functions. In addition, an explicit formula for the strict second subderivative of decomposable functions is derived if the outer support set is sufficiently regular. This yields an alternative characterization of the strong metric regularity of the subdifferential of decomposable functions at a local minimizer. Finally, we verify that the introduced regularity conditions are satisfied by many practical functions and applications.  
\end{abstract}




\section{Introduction and Background}
Composite models provide favorable structural properties that can be leveraged in second-order variational analysis. In such models, the function of interest, $\vp:\R^n\to\Rex$, is assumed to be a local composition of two mappings $\vp_d$ and $F$ around some point $\bar x \in \Rn$:
\begin{equation}
  \label{composition}
  \vp(x)=\vp_d(F(x)) \quad \forall~x\in U. 
\end{equation} 
Here, $U$ is a neighborhood of $\bar x$, the outer function $\vp_d$ is typically lower semicontinuous (lsc) and relatively simple, the inner mapping $F$ is assumed to be sufficiently smooth, and an appropriate constraint qualification is supposed to hold. The advantages of such a composite structure are twofold. Theoretically, the more intricate nonsmoothness of $\vp$ can be transferred to the often simpler function $\vp_d$ facilitating the analysis of its variational properties, \cite{mohammadi2020variational,mohammadi2020twice,mohammadi2022variational}. Practically, such a structure can be utilized to design fast and robust algorithms including, e.g.,  bundle and level methods \cite{sagastizabal2013composite,lan2011level} and prox-linear-type methods \cite{Bur85,LewWri16,DruPaq19}. 

Among these models, perhaps the most well-known one is the concept of fully amenable functions \cite{rockafellar1988first}, where the outer function $\vp_d$ is assumed to be piecewise linear-quadratic. Full calculus for the second subderivative and twice epi-differentiability have been established for fully amenable functions. However, since the subdifferential of a fully amenable function must be a polyhedral set  (cf. \cite[Exercise 10.25(a)]{rockafellar2009variational}), this model does not cover various important applications such as second-order cone and semidefinite programming or nuclear norm minimization. 

In this work, we focus on the class of decomposable functions proposed by Shapiro in \cite{shapiro2003class}, where $\vp_d$ is assumed to be sublinear. As pointed out in \cite{shapiro2003class}, decomposable mappings belong to the category of strongly amenable functions \cite{rockafellar2009variational}, but are not necessarily fully amenable. Decomposable functions are also partly smooth in the sense of Lewis \cite{lewis2002active} provided that the constraint nondegeneracy condition holds, see \cite{shapiro2003class,hare2006functions}. The second-order variational properties of decomposable functions are analyzed in detail in the second author's thesis \cite{milzarek2016numerical}, where  the class of decomposable functions is further divided into decomposable and fully decomposable functions depending on the underlying constraint qualification. Moreover, twice epi-differentiability of $C^2$-fully decomposable functions is established in \cite{milzarek2016numerical}. This result was then recovered in \cite{mohammadi2020twice} through a more general analysis framework that is based on parabolic regularity.   

The goal of this paper is to study the strict twice epi-differentiability and the strict second subderivative of fully decomposable functions. Strict twice epi-differentiability was proposed in \cite{PolRoc96} and was discussed for nonlinear programming and minimax problems in \cite{rockafellar1995second}. Recently, a systematic analysis of the strict twice epi-differentiability of composite functions is conducted in \cite{hang2024chain} for a subclass of fully decomposable functions, where the mapping $\vp_d$ is assumed to be polyhedral. In particular, strict twice epi-differentiability is shown to be equivalent to the strict complementarity condition for such composite functions and an explicit formula for the strict second subderivative is established. We are interested in generalizing these results to decomposable functions, as this class allows covering a broader family of conic problems that do not necessarily have a polyhedral structure -- as required in \cite{hang2024chain}. 

\subsection{Contributions}
In the following, we summarize our main contributions:
\begin{itemize}
  \item  We provide a lower bound for the strict second subderivative of $C^2$-strictly decomposable functions and show that equality in the bound can be  attained along directions in the affine hull of the critical cone. This serves as a strict variant of the calculus rule in \cite[Theorem 4.1 and Corollary 4.3]{benko2022second} and is the theoretical basis of this paper.
  \item Using the obtained chain rule for the strict second subderivative, we establish the equivalence between the strict twice epi-differentiability, the strict complementarity condition, and the continuous differentiability of the proximity operator for $C^2$-strictly decomposable functions\footnote{In parallel and independently, the work \cite{hang2023smoothness} established a similar result for $C^2$-fully decomposable functions, which requires a more restrictive constraint qualification. In addition, in \cite{hang2023smoothness}, decomposability is required to hold on a neighborhood of the point of interest. Such condition is not needed in our analysis; see \Cref{sec:charc-strict} for further discussion.}. As a direct application, this allows us to fully characterize the notion of active strict saddle points, \cite{DavDru22, davis2022escaping}, for strictly decomposable functions.  
  \item We develop a new regularity condition for $\vp_d$ -- the $\usotp$-property -- that allows calculating the strict second subderivative. To the best of our knowledge, computational results for the strict second subderivative are rare and fundamental advances have only been obtained recently in \cite{hang2024chain}. Our analysis goes beyond the results in \cite{hang2024chain} and is applicable to a vast class of conic-type composite problems. Using the obtained representation of the strict second subderivative, we provide full characterizations of the strong metric regularity of the subdifferential $\partial\vp$ when the $\usotp$-property holds. Finally, in \Cref{sec:exam}, we demonstrate applicability of our techniques for various examples, including slices of the second-order cone, matrix intervals, and the Ky Fan $k$-norm. 
\end{itemize}
 
\subsection{Organization}
This paper is organized as follows. In \Cref{sec:pre}, we introduce basic variational tools, fix the notation, and several useful proximal properties. In \Cref{sec:dec}, we provide a formal definition of decomposability and discuss various first- and second-order variational properties of decomposable functions. In particular, we derive chain rules for the standard and strict second subderivative and establish characterizations of (strict) twice epi-differentiability. These results are then utilized to verify the equivalence between the continuous differentiability of the proximity operator and the strict twice epi-differentiability for $C^2$-strictly decomposable functions. In \Cref{sec:uniform-d}, we define the $\usotp$-property which allows us to derive the strict second subderivative in explicit form. Finally, in \Cref{sec:exam}, we present several illustrating examples and applications.

\section{Preliminaries and Variational Tools} \label{sec:pre}
 
\subsection{Notation and Variational Tools}
\textit{General Notation.} Our terminologies and notations follow the standard textbooks \cite{rockafellar1970convex,bonnans2013perturbation,rockafellar2009variational}. By $\iprod{\cdot}{\cdot}$ and $\|\cdot\|$, we denote the standard Euclidean inner product and norm. For matrices, the norm $\|\cdot \|$ is the spectral norm. The sets of symmetric and symmetric positive definite $n \times n$ matrices are denoted by $\mathbb S^n$ and $\mathbb{S}_{++}^n$, respectively. For two matrices $A, B \in \mathbb S^n$, we write $A\succeq B$ if $A-B$ is positive semidefinite. 

\textit{Sets.} We utilize the standard Painlev\'{e}-Kuratowski convergence for sets. Specifically, for a sequence $\{C_k\}_k \subseteq \R^n$, the outer and inner limit are defined via 
\[ {\limsup}_{k\to\infty} C_k := \{x\in\R^n: \exists~ \{j_k\}_k \subseteq \n,~x_{j_k}\in C_{j_k},~x_{j_k}\to x\} \]
and $\liminf_{k\to\infty} C_k:=\{x\in\R^n: \exists~x_{k}\in C_{k},~x_{k}\to x\}$. 
We write $C_k\to C$ when both limits coincides. 
For any closed set $C \subseteq \Rn$, we use $\Pi_C : \Rn \rightrightarrows \Rn$ to denote the projection mapping onto $C$. If $C$ is closed and convex, then $\Pi_C$ reduces to a single-valued mapping. We will identify $\Pi_C$ with the corresponding projection matrix if $C$ is a linear subspace. We further use $\iota_C$ and $\sigma_C$ to denote the indicator and support function of $C$. The set $\aff(C)$ denotes the affine hull of $C$ and the relative interior of a convex set $C \subseteq \Rn$ is given by $\mathrm{ri}(C)$. The tangent cone $T_{C}(x)$ for a set $C$ at $x\in C$ is defined as $T_{C}(x)=\limsup_{t\downarrow 0}\frac{C-x}{t}$. Since we only consider the normal cone of convex sets in this paper, the normal cone $N_{C}(x)$ of $C$ at $x$ can be defined as the polar of the tangent cone, i.e., $N_C(x) = T_C(x)^\circ$. Moreover, for a closed convex cone $K$, the lineality space $\lin(K)$ of $K$ is given by $\lin(K)=K\cap -K$. 


\textit{Epi-Convergence and Derivatives.} The effective domain of a function $\theta : \Rn \to \Rexx$ is defined as $\dom(\theta):=\{x \in \Rn : \theta(x)<\infty\}$. We call $\theta$ proper if $\theta(x) > -\infty$ for all $x$ and $\dom(\theta) \neq \emptyset$. The epigraph $\epii(\theta)$ of a mapping $\theta:\R^n\to\Rex$ is given by $\epii(\theta):=\{(x,t)\in\R^n\times \R: \theta(x) \leq t\}$. We say that a sequence of extended real-valued functions $\{\theta_k\}_k$ epi-converges to $\theta$ if $\epii(\theta_k)\to\epii(\theta)$ (in the Painlev\'{e}-Kuratowski sense). We will write $\theta_k\elimit\theta$ in such case. Epi-convergence of the sequence $\{\theta_k\}_k$ to the function $\theta$ is equivalent to the condition
\be \label{eq:def-epi} \left[ \begin{array}{ll} \liminf_{k \to \infty}~\theta_k(x^k) \,\, \geq \theta(x) & \text{for every sequence } x^k \to x, \\[.5ex] \limsup_{k \to \infty}~\theta_k(x^k) \leq \theta(x) & \text{for some sequence } x^k \to x, \end{array} \right. \quad \forall~x, \ee
cf. \cite[Proposition 7.2]{rockafellar2009variational}. For an extended real-valued function $\theta:\R^n\to \Rex$ and  $x\in\dom(\theta)$, the regular subdifferential of $\theta$ at $x$ is defined as:
\[  \hat{\partial}\theta(x):=\{v\in \R^n: \theta(y)\geq \theta(x)+\langle  v,y-x\rangle+o(\|y-x\|)\}.      \] 
The subdifferential of $\theta$ at $x$ is then given by $ \partial \theta(x) := \limsup_{y\to x, \theta(y)\to\theta(x)} \hat{\partial}\theta(y)$. The {lower subderivative} of $\theta$ at $x \in \dom(\theta)$ in the direction $h \in \Rn$ is defined as follows: 
\[ \rd\theta(x)(h) := \liminf_{t \downarrow 0, \, \tilde h \to h}~\Delta_t \;\! \theta(x)(\tilde h), \quad \Delta_t \;\! \theta(x)(h) := \frac{\theta(x+th)-\theta(x)}{t}. \]
We say that $\theta$ is {(directionally) epi-differentiable} at $x$ in the direction $h \in \Rn$ with {epi-derivative} $\rd\theta(x)(h)$ iff for every $\{t_k\}_k$, $t_k \downarrow 0$, the sequence $\{\Delta_{t_k} \;\! \theta(x)\}_k$ epi-converges to $\rd\theta(x)(h)$ in the sense of \cref{eq:def-epi} for fixed $h$.
The function $\theta$ is called {directionally differentiable} at $x$ in the direction $h$ if the limit $\theta^\prime(x;h) = \lim_{t \downarrow 0} \Delta_t \;\! \theta(x)(h)$ exists. Moreover, $\theta$ is {semidifferentiable} at $x$ in the direction $h \in \Rn$ if the limit $\lim_{t \downarrow 0, \tilde h \to h} \Delta_t \;\! \theta(x)(\tilde h)$ exists. In this case, we use the same term $\theta^\prime(x;h) $ to denote its limit. 
In this paper, we utilize different notions of generalized second-order differentiability to study the variational properties of the model \cref{composition}. We mainly work with the {second subderivative} $\mathrm{d}^2\theta(x|v)(h)$ of $\theta$ at $x$ relative to $v$ in the direction $h \in \Rn$, which is given by
%
\[ \mathrm{d}^2\theta(x|v)(h) = \liminf_{t\downarrow 0, \, \tilde h \to h}~\Delta_t^2  \;\! \theta(x|v)(\tilde h), \quad \Delta_t^2  \;\! \theta(x|v)(h) := \frac{\theta(x+th)-\theta(x)-t \cdot \iprod{v}{h}}{\half t^2}. \]
We say that $\theta$ is {twice epi-differentiable} at $x$ for $v \in \Rn$ if the second-order difference quotients $h \mapsto \Delta_{t_k}^2  \;\! \theta(x|v)(h)$ epi-converge in the sense of \cref{eq:def-epi} for all $\{t_k\}_k$ with $t_k \downarrow 0$. We use $\mathrm{d}^2\theta(x|v)$ to denote the corresponding epi-limit. In addition, $\theta$ is called {properly twice epi-differentiable} at $x$ for $v$ if $\mathrm{d}^2\theta(x|v)$ is proper. 
%
%
Finally, $\theta$ is called {twice semidifferentiable} at $x$ if it is semidifferentiable and if the limit 
\[ \lim_{t\downarrow 0, \, \tilde h \to h} \frac{\theta(x+t\tilde h) - \theta(x) - t \cdot \theta^\prime(x;\tilde h)}{\frac12 t^2} \]
exists for all $h \in \Rn$. The limiting function will then be denoted by $\theta^{\prime\prime}(x; \cdot)$. The interested reader is referred to \cite[Chapter 13]{rockafellar2009variational} and \cite[Sections 2.2 and 3.3.5]{bonnans2013perturbation} for a thorough discussion of these second-order concepts. 

For a mapping $F:\mathbb{R}^{n}\rightarrow\mathbb{R}^m$, the set $\mathcal{R}(F):=\{y \in \R^m: \exists~x \in \Rn \, \text{with} \, y=F(x)\}$ denotes the range of $F$. For a locally Lipschitz continuous mapping $F:\R^n\to \R^m$, the Bouligand and Clarke subdifferential are defined as $\partial_B F(x) := \{V\in \R^{m\times n}: x^k\to x,~\text{$F$ is differentiable at $x^k$},~ DF(x^k)\to V\}$ and $\partial F(x) := \conv(\partial_B F(x))$, respectively. The graph of a set-valued mapping $F : \R^{n}\rightrightarrows\R^{m}$ is denoted by $\gph(F)$.
%


\subsection{Basic Proximal Properties}
\label{sec:basicprox}
In the following, we collect basic properties of the proximity operator and the Moreau envelope that will play a more important role in our later analysis. For a proper, lsc function $\vp:\Rn \to \Rex$ and $\tau > 0$, the proximity operator and Moreau envelope of $\vp$ are defined as

\[\proxs : \Rn \rightrightarrows \Rn, \quad \proxs(x) \in \argmin_{y \in \Rn}~\vp(y) + \frac{1}{2\tau}\|y-x\|^2 \]
and $\envs(x) := e_\tau(\varphi)(x) := \inf_{y \in \Rn} \vp(y) + \frac{1}{2\tau}\|y-x\|^2$, see \cite{moreau1965proximite,rockafellar2009variational,bauschke2017convex}. Throughout this work, $\vp$ is assumed to be a proper, lsc mapping and the parameter $\tau$ in $\proxs$ is positive. We first introduce a regularity concept for the mapping $\vp$.

\begin{defn}
    \label{defn2-1}
     We say $\vp$ is globally prox-regular at $\bar x$ for $\bar v \in \partial\vp(\bar x)$ 
     (with constant $\rho$), if $\vp$ is locally lsc at $\bar x$, and there exists $\veps>0$ and $\rho\geq 0$ such that 
    \[ \vp(x')\geq \vp(x)+\langle v,x'-x \rangle-\frac{\rho}{2}\|x'-x\|^2  \quad \forall~x'\in\R^n         \]
     when $\|x-\bar x\| < \veps$, $\|v-\bar v\|<\veps$, $v\in \partial\vp(x)$, and $\vp(x)<\vp(\bar x)+\veps$. If this is true for all $\bar v\in \partial \vp(\bar x)$, then we say $\vp$ is globally prox-regular at $\bar x$.
  \end{defn}

Global prox-regularity is not a very restrictive assumption, i.e., it is guaranteed to hold under prox-boundedness and prox-regularity, cf. \cite[Propositions 8.46(f) and 13.37]{rockafellar2009variational}. We mainly assume this condition to ensure local single-valuedness and Lipschitz continuity of the proximity operator and to have more explicit control of $\tau$.

\begin{prop}
    \label{prop_equi_quad_diff}
    Assume that $\vp$ is globally prox-regular at $\bar x$ for $\bar v\in\partial \vp(\bar x)$ with constant $\rho \geq 0$ and $\tau\rho<{1}$. The following statements are equivalent:
\begin{enumerate}[label=\textup{\textrm{(\roman*)}},topsep=0pt,itemsep=0ex,partopsep=0ex]
        \item The function $\vp$ is properly twice epi-differentiable at $\bar x$ for $\bar v$ and $\rd^2\vp(\bar x|\bar v)$ is generalized quadratic.
        \item The function $\envs$ is twice differentiable at $\bar z =\bar x+\tau \bar v$ and $e_\tau(\frac{1}{2}\rd^2\vp(\bar x|\bar v))$ is a quadratic function.
        \item The proximity operator $\proxs$ is differentiable at $\bar z=\bar x+\tau \bar v$.
    \end{enumerate}
\end{prop}

\cref{prop_equi_quad_diff} is a straightforward variant of \cite[Theorems 3.8 and 3.9]{poliquin1996generalized}. We provide a short proof in \cref{app:sec:equi-epi}. Here, $h : \Rn \to \Rex$ is called generalized quadratic, if there exist a symmetric matrix $Q \in \R^{n \times n}$ and a linear subspace $K \subseteq \Rn$ such that $h(x) = \iprod{x}{Qx} + \iota_K(x)$, cf. \cite[Definition 3.7]{poliquin1996generalized}. We continue with several basic properties of the proximity operator and Moreau envelope. 

\begin{prop}
    \label{prop-lowerbound-Q} Let $\vp$ be globally prox-regular at $\bar x$ for $\bar v\in\partial \vp(\bar x)$ with constant $\rho \geq 0$ and $\tau\rho<{1}$. Then, there is a neighborhood $U$ of $\bar z := \bar x+\tau \bar v$ such that the following conditions hold:
\begin{enumerate}[label=\textup{\textrm{(\roman*)}},topsep=0pt,itemsep=0ex,partopsep=0ex]
        \item The proximity operator $\proxs$ is single-valued and Lipschitz continuous with modulus $\frac{1}{1-\tau\rho}$ on $U$. Moreover, $\envs$ is continuously differentiable on $U$ with $\nabla\envs(z) = \frac{1}{\tau}(z-\proxs(z))$ for all $z \in U$.
        \item For all $z\in U$ and $V\in\partial\proxs(z)$, it holds that $0\preceq V \preceq \frac{1}{1-\tau\rho}I$. 
    \end{enumerate}
\end{prop}
\begin{proof}
     Part (i) follows from \cite[Theorem 4.4]{PolRoc96} or \cite[Proposition 13.37]{rockafellar2009variational}. To prove part (ii), let us introduce the set $S:=\{B\in\bS^{n}: 0 \preceq B\preceq {(1-\tau\rho)^{-1}}I \}$. Since $S$ is convex, it suffices to show $\partial_B\proxs(z)\subseteq S$. Let $z\in U$ be arbitrary and suppose that $\proxs$ is differentiable at $z$. Let us set $x=\proxs(z)$ and $v=\tau^{-1}{(z-x)}\in\partial\vp(x)$, cf. \cite[Proposition 4.3]{PolRoc96}. Utilizing part (i) and the definition of prox-regularity, we can infer that $\vp$ is globally prox-regular at $x$ for $v$ with the same constant $\rho \geq 0$ (after potentially shrinking $U$). By \cref{prop_equi_quad_diff}, we then have $\rd^2\vp(x|v)(h)=\langle h, Qh \rangle+\iota_{K}(h)$ for some linear subspace $K$ and symmetric matrix $Q$. Without loss of generality, we may additionally assume $\cR(Q)\subseteq K$ (otherwise, we can redefine $Q$ as $Q := \Pi_KQ\Pi_K$). Applying \cite[Theorem 13.36 and Proposition 13.49]{rockafellar2009variational}, it holds that $Q+\rho I\succeq 0$. Combining this with \cite[Exercises 13.18, 13.35, and 13.45]{rockafellar2009variational}, we obtain
     \begin{equation} \label{eq:use-in-app}    D\proxs(z)h=\argmin_{s\in\R^n}~\frac{1}{2}\rd^2\vp(x|v)(s)+\frac{1}{2\tau}\|s-h\|^2=\Pi_{K}(I+\tau Q)^{-1}\Pi_{K}h,            \end{equation}
     which implies $D\proxs(z)=\Pi_{K}(I+\tau Q)^{-1}\Pi_{K}$. Due to $I+\tau Q = \tau (Q+\rho I) + (1-\tau\rho) I \succeq (1-\tau\rho) I$, the matrix $(I+\tau Q)^{-1}$ is positive definite and hence $D\proxs(z)$ is positive semidefinite. Furthermore, we have 
     \[     \|D\proxs(z)\|\leq \|\Pi_K\|^2\|(I+\tau Q)^{-1}\|\leq (1-\tau\rho)^{-1}.                   \]
     This proves $D\proxs(z)\in S$. The claim $\partial_B\proxs(z)\subseteq S$, $z \in U$, then follows by a density argument and from the fact that $S$ is closed.
\end{proof}

The following result generalizes \cite[Theorem 7.37]{rockafellar2009variational} to the weakly convex case. \cref{prop2-5} will be applied later to study epi-convergence of a sequence of second subderivatives (which, by \cite[Proposition 13.49]{rockafellar2009variational}, are weakly convex). Here, an extended real-valued function $h$ is said to be $\rho$-weakly convex, if the mapping $h+\frac{\rho}{2}\|\cdot\|^2$ is convex. The proof of \cref{prop2-5} is presented in  \cref{app:sec:prop-technical}.

 
\begin{prop}
    \label{prop2-5}
    Let $\{f_k\}_k$ and $f$ be a family of proper, lsc, and $\rho$-weakly convex functions. 
    Suppose that $f_k+\frac{\rho}{2}\|\cdot\|^2\geq 0$ and $f+\frac{\rho}{2}\|\cdot\|^2\geq 0$ for all $k \in \n$. Then, the following statements are equivalent:
\begin{enumerate}[label=\textup{\textrm{(\roman*)}},topsep=0pt,itemsep=0ex,partopsep=0ex]
        \item It holds that $f_k\elimit f$.
        \item We have $\mathrm{env}_{\tau f_k} \rightarrow \mathrm{env}_{\tau f}$ point-wisely for some $\tau>0$ with $\tau\rho<1$.  
    \end{enumerate}
\end{prop}


\section{Decomposable Functions}
\label{sec:dec}
In this section, we conduct a systematic study of the second-order variational properties of decomposable functions. We first provide a formal definition of decomposablility and list several basic properties in \Cref{sec:def-basic}. In \Cref{sec:full-chain}, we focus on the class of $C^2$-strictly decomposable functions. Specifically, we derive a chain rule for the (strict) second subderivative of $C^2$-strictly decomposable functions. We then leverage the obtained results to prove the equivalence between the strict twice epi-differentiability at one point and the continuous differentiability of the proximity operator generalizing \cite[Theorems 3.5 and 4.9]{hang2024chain}. Under an additional geometric regularity condition on the support set $\pvpd = \partial\vp_d(0)$, the so-called $\usotp$-property, we further provide an explicit formula for the strict second subderivative. As applications, these results allow us to fully characterize the strict saddle point property studied in \cite{DavDru22} and the strong metric regularity of $\partial\vp$. 
 
As mentioned, the concept of {decomposable functions} was initially proposed by Shapiro in \cite{shapiro2003class} and is strongly related to the notions of {amenable functions}, see \cite{PolRoc92,PolRoc93} or \cite[Chapter 10.F]{rockafellar2009variational}, and of {$C^{\ell}$-cone reducible sets} in nonlinear, constrained optimization, see \cite{bonnans2013perturbation} and \Cref{def:cone-red}. The class of decomposable functions is rich and a vast number of optimization problems can be treated within the framework of decomposable functions \cite{shapiro2003class,milzarek2016numerical}, including polyhedral and group sparse problems, second-order cone and semidefinite programming or nuclear norm regularized optimization problems. 

\begin{defn}[Decomposable functions] \label{def:decomp}
  A function $ \varphi: \mathbb{R}^{n} \rightarrow\Rex $ is called $ C^{\ell} $-decomposable, $ \ell \in \mathbb{N} $, at a point $ \bar{x} \in \dom(\varphi)$, if there is an open neighborhood $ U $ of $ \bar{x} $ such that
  %
  \[ \varphi(x)=\varphi(\bar{x})+\varphi_{d}(F(x)), \quad \forall~x \in U, \]
  %
  and the functions $ \varphi_{d} $ and $ F $ satisfy:
\begin{enumerate}[label=\textup{\textrm{(\roman*)}},topsep=0pt,itemsep=0ex,partopsep=0ex]
	\item $ F: U \rightarrow \mathbb{R}^{m} $ is $\ell$-times continuously differentiable on $ U $ and $ F(\bar x)=0 $.
	\item $ \varphi_{d}: \mathbb{R}^{m} \rightarrow \Rex$ is convex, proper, lsc, and positively homogeneous, or equivalently, we have $\vp_d=\sigma_{\pvpd}$ for some nonempty closed convex set $\pvpd \subseteq \R^m$.  
	\item Robinson's constraint qualification holds at $ \bar{x} $:
	\end{enumerate}
  \begin{equation} \label{eq:robinson}
  0 \in \mathrm{int}\{F(\bar{x})+D F(\bar{x}) \mathbb{R}^{n}-\dom(\varphi_{d})\}=\mathrm{int}\{D F(\bar{x}) \mathbb{R}^{n}-\dom(\varphi_{d})\}.
  \end{equation}
We say that $ \varphi $ is $ C^{\ell} $-strictly decomposable at $\bar{x}$ (for $\bar\lambda \in \pvpd$) if $ \varphi $ is $ C^{\ell} $-decomposable at $ \bar{x} $ and if, in addition, the strict condition
\begin{equation} \label{eq:strict-cq} D F(\bar{x}) \mathbb{R}^{n}-N_{\pvpd}(\bar\lambda)=\mathbb{R}^{m} \end{equation}
is satisfied at $ \bar{x} $ for some $ \bar\lambda \in \pvpd $. We say that $ \varphi $ is $C^{\ell}$-fully decomposable at $ \bar{x} $ if the strict condition \cref{eq:strict-cq} can be replaced by the nondegeneracy condition:
\begin{equation} \label{eq:nondeg-cq} D F(\bar{x}) \mathbb{R}^{n}+\lin(N_{\pvpd}(\bar\lambda))=\mathbb{R}^{m}. \end{equation}
%
\end{defn}

If $\vp$ is decomposable at $\bar x$, then the functions $\vp_d$ and $F$ form a \textit{decomposition pair} $(\vp_d,F)$ of $\vp$. Of course, decomposition pairs do not need to be unique. 

Due to \cite[Proposition 2.97 and Corollary 2.98]{bonnans2013perturbation}, Robinson's constraint qualification \cref{eq:robinson} is equivalent to the basic constraint qualification used in the definition of amenable functions, cf. \cite[Definition 10.23]{rockafellar2009variational}. Therefore, any $C^1$-decomposable function is also amenable. As discussed in \cite[Lemma 5.1.7]{milzarek2016numerical}, the strict constraint qualification is in analogy to \cite[Equation (4.119)]{bonnans2013perturbation} and allows establishing the uniqueness of the associated Lagrange multipliers. It will play a similar role in the case of decomposable functions, see \cref{prop:prop-strict}. The constraint nondegeneracy condition appeared originally in Shapiro's work \cite{shapiro2003class} to ensure partial smoothness of decomposable functions and hence, $C^2$-fully decomposable functions are $C^2$-partly smooth in the sense of \cite[Definition 14]{DanHarMal06}. We summarize these observations in the following fact.

\begin{fact} \label{fact:structure} Every $C^2$-decomposable function (at $\bar x)$ is strongly amenable (at $\bar x$). Every $C^2$-fully decomposable mapping (at $\bar x$) is $C^2$-partly smooth (at $\bar x$).
\end{fact}

\subsection{Basic Properties}
\label{sec:def-basic}

In this section, we present stability results and several basic properties of decomposable functions. 
In particular, we show that $C^2$-decomposable functions are properly twice epi-differentiable and we provide an explicit formula for the second-order subderivative. 



The following preparatory result allows deriving equivalent representations of the constraint qualifications utilized in \cref{def:decomp} and other works. 
\begin{prop}
    \label{nprop4-3}
    Let $A,B\subseteq \R^m$ be closed convex cones. Then, $A+B=\R^m$ iff $A^\circ\cap B^\circ=\{0\}$. 
\end{prop}
\begin{proof} By \cite[Proposition 6.27]{bauschke2017convex} and taking the polar cone on both sides of $A+B=\R^m$, we obtain $A^\circ\cap B^\circ=\{0\}$ . Next, let us assume $A^\circ\cap B^\circ=\{0\}$. Taking the polar cone on both sides and using \cite[Proposition 6.35]{bauschke2017convex}, it follows $\mathrm{cl}(A+B)=\R^m$. Since the set $A+B$ is convex, it has nonempty relative interior and we can infer $A+B=\R^m$, cf. \cite[Theorems 6.1, 6.2, and 6.3]{rockafellar1970convex} or \cite[Proposition 2.97]{bonnans2013perturbation}.
\end{proof}

We now show that the constraint qualifications in \cref{def:decomp} are stable under small perturbations.
\begin{prop}
    \label{prop1-11} We consider the setting in \cref{def:decomp} with $F$ being continuously differentiable. It holds that: 
\begin{enumerate}[label=\textup{\textrm{(\roman*)}},topsep=0pt,itemsep=0ex,partopsep=0ex]
        \item If Robinson's constraint qualification \cref{eq:robinson}  holds at $\bar x$, then it is also satisfied on some neighborhood of $\bar x$.
        \item If the strict condition \cref{eq:strict-cq} is satisfied at $\bar x$ and $\bar\lambda \in \pvpd$, then it holds that $DF(x)\Rn - N_\pvpd(\bar\lambda) = \Rm$ for all $x$ in a neighborhood of $\bar x$. 
        \item If the nondegeneracy condition \cref{eq:nondeg-cq} holds at $\bar x$, then it holds on some neighborhood of $\bar x$ and for every $\lambda\in \pvpd$.
    \end{enumerate}
\end{prop}
 \begin{proof}
     The proof of (i) can be found in \cite[Section 2.3.4 and Remark 2.88]{bonnans2013perturbation}. For (iii), notice that the nondegeneracy condition is equivalent to $\ker(DF(\bar x)^\top)\cap \parr(\pvpd)=\{0\}$ by \cref{nprop4-3}, which is independent of the choice of $\lambda\in\pvpd$. Here, $\parr(\pvpd)$ denotes the linear subspace of $\R^m$ parallel to $\aff(\pvpd)$. Following the standard perturbation arguments as in \cite[Section 2.3.4 and Remark 2.88]{bonnans2013perturbation}, the condition $ \ker(DF(x)^\top)\cap \parr(\pvpd)=\{0\}$ holds on some neighborhood of $\bar x$. The argument for (ii) is similar. 
 \end{proof}
 
The following proposition provides a chain rule for the subdifferential of $\vp$.
\begin{prop}
    \label{prop4-4}
    Let $\vp$ be $C^2$-decomposable at $\bar x$ with decomposition pair $(\vp_d,F)$. Then there exists a neighborhood $U$ of $\bar x$ such that for all $x\in U\cap\dom(\vp)$, we have:
\begin{enumerate}[label=\textup{\textrm{(\roman*)}},topsep=0pt,itemsep=0ex,partopsep=0ex]
      \item It holds that $\partial\vp(x)=DF(x)^\top\partial\vp_d(x)$.
      \item The function $\vp$ is prox-regular and subdifferentially continuous at $x$ relative to $\dom(\vp)$.
    \end{enumerate}
\end{prop}
\begin{proof}
    By \cref{fact:structure} and \cite[Exercise 10.25 (b)]{rockafellar2009variational}, there is a neighborhood $U$ of $x$ such that $\vp$ is strongly amenable on $U\cap\dom(\vp)$. We can then apply \cite[Exercise 10.25 (a) and Theorem 10.6]{rockafellar2009variational} to get the desired chain rule for every $x\in U\cap \dom(\vp)$. The statements in (ii) follow from \cite[Proposition 13.32]{rockafellar2009variational}. 
\end{proof}
 
Following the discussion in \Cref{sec:basicprox}, \cref{prop4-4}~(ii) implies that the $C^2$-decomposable mapping $\vp$ is globally prox-regular (at $\bar x$) if $\vp$ is prox-bounded. 

Strict decomposability allows a finer characterization of the subdifferential of $\vp$. Specifically, it implies that the set 
\[ \Lambda(\bar x,\bar v) := \{\lambda\in\partial\vp_d(F(\bar x)): DF(\bar x)^\top\lambda=\bar v\}, \quad \bar v \in \partial \vp(\bar x), \] 
is a singleton and the set-valued mapping $\Lambda: \Rn \times \Rn \rightrightarrows \Rm$ is outer Lipschitz continuous around $(\bar x,\bar v)$. 

\begin{prop} \label{prop:prop-strict} Let $\vp$ be $C^2$-strictly decomposable at $\bar x$ for $\bar \lambda\in \pvpd$ and $\bar v=DF(\bar x)^\top \bar\lambda$. Then, it holds that:
 \begin{enumerate}[label=\textup{\textrm{(\roman*)}},topsep=0pt,itemsep=0ex,partopsep=0ex]
      \item The point $\bar\lambda$ is the only element in $\pvpd$ satisfying $DF(\bar x)^\top \lambda=\bar v$, i.e., we have $\Lambda(\bar x,\bar v) = \{\bar\lambda\}$.
      \item The mapping $\Lambda: \Rn \times \Rn \rightrightarrows \Rm$ is outer Lipschitz continuous at $(\bar x,\bar v)$ relative to $\gph(\partial\vp)$.
    \end{enumerate}
\end{prop}

As mentioned, condition \cref{eq:strict-cq} can be interpreted as a variant of the strict (Robinson) constraint qualification. \cref{prop:prop-strict} then essentially follows from \cite[Proposition 4.47]{bonnans2013perturbation} and its derivation is similar to \cite[Proposition 2.2]{hang2024chain}.

We now discuss the twice epi-differentiability of $C^2$-decomposable functions and provide a characterization of the second subderivative. For $C^2$-fully decomposable functions, the results in \cref{prop:soc:decomp-subquad} first appeared in the second author's thesis \cite[Lemma 5.3.27]{milzarek2016numerical}. The generalization to $C^2$-decomposable functions is possible due to novel techniques proposed in \cite{mohammadi2020twice} that is based on parabolic regularity. A sketch of the proof can be found in \cite[Remark 5.6 (d)]{mohammadi2020twice}. Although the derivation in \cite{mohammadi2020twice} is stated for $C^2$-fully decomposable functions, the constraint qualification \cite[Equation (4.6)]{mohammadi2020twice} is equivalent to \cref{eq:robinson}, and thus it is also applicable for $C^2$-decomposable functions.

\begin{prop} \label{prop:soc:decomp-subquad} Suppose that $\vp$ is $C^2$-decomposable at $\bar x$ with decomposition pair $(\vp_d,F)$. Then, $\vp$ is properly twice epi-differentiable at $\bar x$ relative to $\bar v\in\partial\vp(\bar x)$ and the second order subderivative of $\vp$ at $\bar x$ is given by
\be    \rd^2\vp(\bar x|\bar v)(h)=\max_{\lambda \in \Lambda(\bar x,\bar v)}~\iprod{\lambda}{D^2F(\bar x)[h,h]}+\iota_{\mathcal{C}(\bar x,\bar v)}(h) \quad \forall~h \in \Rn, \ee
where $\Lambda(\bar x,\bar v) = \{\lambda\in\partial\vp_d(F(\bar x)): DF(\bar x)^\top\lambda=\bar v\}$ and the set $\mathcal C{(\bar x,\bar v)}$ is the critical cone $\mathcal C{(\bar x,\bar v)}:=\{h\in\R^n:\rd\vp(\bar x)(h)=\langle \bar v,h \rangle\}$.
\end{prop}

\subsection{Chain Rules for Strictly Decomposable Functions}
\label{sec:full-chain}

Our primary goal in the following sections is to study and discuss \emph{strict} twice epi-differentiability of decomposable functions. Strict twice epi-differentiability is a stronger variational second-order property that was first introduced in \cite{rockafellar1995second,PolRoc96}. In particular, the mapping $\vp$ is said to be strictly twice epi-differentiable at $\bar x$ for $\bar v \in \partial \vp(\bar x)$ if the difference quotients $\Delta_t^2  \;\! \vp(x|v)$ epi-converge to a function as $t \downarrow 0$, $\gph(\partial \vp) \ni (x,v) \to (\bar x,\bar v)$ with $\vp(x) \to \vp(\bar x)$. The strict (lower) second-order subderivative is then given by
\[ \mathrm{d}^2_s\vp(\bar x|\bar v)(h) := \liminf_{\begin{subarray}{c} t\downarrow 0, \, \tilde h \to h, \, \vp(x) \to \vp(\bar x), \\ \gph(\partial \vp) \ni (x,v) \to(\bar x,\bar v) \end{subarray}}~\Delta_t^2  \;\! \vp(x|v)(\tilde h). \]
Since decomposable functions are subdifferentially continuous (cf. \cref{prop4-4}), we will often drop the condition $\vp(x) \to \vp(\bar x)$ in the definition of the strict second subderivative. Computational results on strict epi-differentiability are generally rare and limited to max-type functions \cite{PolRoc96}, polynomial functions \cite{hang2024chain}, and polynomial composite-type problems \cite{hang2024chain}. Motivated by the recent and pioneering observations in \cite{hang2024chain}, we now investigate strict twice epi-differentiability for a class of strictly decomposable functions which allows covering problems with non-polyhedral structures. 
Different from the results in \cref{prop:soc:decomp-subquad}, strict epi-differentiability relies on the behavior of $\vp$ in a neighborhood of $(\bar x,\bar v)$. As a first preparatory step, we present a general chain rule-type result for the strict second subderivative of a strictly decomposable function.


\begin{thm}
  \label{thm-strict-chain}
  Let $\vp$ be $C^2$-strictly decomposable at $\bar x$ for $\bar \lambda$ with decomposition pair $(\vp_d, F)\equiv(\sigma_{\pvpd},F)$. It holds that
  \begin{equation} \label{eq:strict-chain}
  \rd^2_s\vp(\bar x|DF(\bar x)^\top\bar \lambda)(h)\geq\rd^2_s\vp_d(0|\bar\lambda)(DF(\bar x)h)+\langle \bar \lambda,D^2F(\bar x)[h,h]\rangle \quad \forall~h \in \Rn.
  \end{equation}
  Furthermore, if $h\in  DF(\bar x)^{-1}\aff(N_{\pvpd}(\bar\lambda))$, then \cref{eq:strict-chain} is satisfied with equality and we have $\rd^2_s\vp_d(0|\bar\lambda)(DF(\bar x)h) = 0$.
\end{thm}



\cref{thm-strict-chain} is a strict variant of the chain rule provided in \cite[Theorem 4.1]{benko2022second}. In particular, the special structure of $\vp_d$ can be leveraged to establish \cref{eq:strict-chain} without requiring the full rank assumption used in \cite[Corollary 4.3]{benko2022second}. The chain rule in \cref{thm-strict-chain} allows us to fully characterize the strict twice epi-differentiability of $C^2$-strictly decomposable functions. Moreover, \cref{eq:strict-chain} is the basis of a general computational formula for the strict second subderivative which will be discussed in detail in \Cref{sec:uniform-d}. We start the derivation of \cref{thm-strict-chain} with a technical result.


\begin{prop}
  \label{exist_ri} Let $F : \Rn \to \Rm$ be continuously differentiable and let $\bar x \in \Rn$ be given with $F(\bar x) = 0$. Suppose that $DF(\bar x)\R^n-\cN=\R^m$ for some closed convex cone $\cN$. Then, there is a sequence $\{x^k\}_k$ with $x^k\to\bar x$, $F(x^k)\in \ri(\cN)$ (for all $k$), and $DF(x^k)^{-1}\aff(\cN)\to DF(\bar x)^{-1}\aff(\cN)$.
\end{prop}

\begin{proof}
  We first argue that $DF(\bar x)\R^n\cap \ri(\cN)\neq \emptyset$. Let $q\in\ri(\cN)$ be arbitrary. By assumption, there exists $w\in \R^n$ and $v\in \cN$ such that $DF(\bar x)w-v=q$. Since $\cN$ is a convex cone, it holds that $\ri(\cN)+\cN=\ri(\cN)$ (this follows from $\frac12 \ri(\cN)+\frac12 \cN \subseteq \ri(\cN) $ and $0 \in \cN$, cf. \cite[Theorem 6.1]{rockafellar1970convex}). This implies $DF(\bar x)w\in \ri(\cN)$. Next, we consider the set $K:=\{x\in\R^n:F(x)\in \cN\}$. The conditions $F(\bar x) = 0$ and $DF(\bar x)\Rn-\cN = \Rm$ imply that Robinson's constraint qualification (cf. \cite[Equation (2.166)]{bonnans2013perturbation}) holds for $K$ at $\bar x$. Hence, we can apply \cite[Corollary 2.91]{bonnans2013perturbation} to obtain $T_{K}(\bar x)=DF(\bar x)^{-1}\cN$. By the previous argument, we may select $h\in \R^n$ with $DF(\bar x)h\in \ri(\cN)$. Then, we have $h\in T_{K}(\bar x)$ and there are $t_k\downarrow 0$ and $h^k\to h$ such that $x^k:=\bar x+t_kh^k\in K$, i.e., $F(x^k)\in \cN$. Applying Taylor expansion, it follows:
  \[ F(x^k)=F(\bar x)+t_kDF(\bar x)h^k+t_kq^k = t_k(DF(\bar x)h+q^k),\quad q^k\to 0.     \]
  Since $F(x^k)\in \cN$, we can infer $q^k=DF(\bar x)h-F(x^k)/t_k\in \aff(\cN)$. Combining this with $DF(\bar x)h\in \ri(\cN)$ and $q^k\to 0$, it holds that $F(x^k)\in \ri(\cN)$ for all sufficiently large $k$. So we may assume $F(x^k)\in \ri(\cN)$ for all $k\in\n$. The result $DF(x^k)^{-1}\aff(\cN)\to DF(\bar x)^{-1}\aff(\cN)$ follows from \cite[Theorem 4.32]{rockafellar2009variational} if $DF(\bar x)\R^n$ can not be separated from $\aff(\cN)$. This, however, is immediate from $DF(\bar x)\Rn-\cN = \Rm$.
\end{proof}

\begin{proof}[Proof of \cref{thm-strict-chain}]
By definition, we can select $x^k\to\bar  x,~\partial\vp(x^k)\ni v^k\to DF(\bar x)^\top\bar\lambda,~t_k\downarrow 0$, and $h^k\to h$, such that 
  \[  \rd^2_s\vp(\bar x|DF(\bar x)^\top \bar\lambda)(h)={\lim}_{k\rightarrow 0}~\Delta_{t_k}^2  \;\! \vp(x^k|v^k)(h^k).
  \]
By \cref{prop4-4} (i), we can express $v^k$ as $v^k=DF(x^k)^\top\lambda^k$ with $\lambda^k\in\partial\vp_d(F(x^k)) $, i.e., $\lambda^k \in \Lambda(x^k,v^k)$. According to \cref{prop:prop-strict}, we know that $\lambda^k\to\bar\lambda$. Therefore, setting $w^k := (F(x^k+t_kh^k)-F(x^k))/t_k$, it holds that
  \begin{align*}
    & \hspace{-2ex} \rd^2_s\vp(\bar x|DF(\bar x)^\top\bar\lambda)(h) \\ 
     & = \lim_{k\rightarrow 0}\frac{\vp_d(F(x^k+t_kh^k))-\vp_d(F(x^k))-t_k\langle \lambda^k,w^k\rangle}{\frac{1}{2}t_k^2} + \frac{\iprod{\lambda^k}{w^k-DF(x^k)h^k}}{\frac12 t_k} \\
     & \geq \rd^2_s\vp_d(0|\bar\lambda)(DF(\bar x)h)+\langle \bar\lambda,D^2F(\bar x)[h,h]\rangle,
   \end{align*}
where we used $F(\bar x) = 0$,  $w^k \to DF(\bar x)h$, and the definition of $\rd^2_s\vp_d(0|\bar\lambda)$.   

Next, let us assume $h\in DF(\bar x)^{-1}\aff(N_{\pvpd}(\bar\lambda))$. By \cref{exist_ri}, we may select $x^k\to \bar x$ with $F(x^k)\in \ri(N_{\pvpd}(\bar\lambda))$ and $h^k\to h$ with $h^k\in DF(x^k)^{-1}\aff(N_{\pvpd}(\bar\lambda))$. Let us set $\cN := N_{\pvpd}(\bar\lambda)$. By \cref{prop1-11} (ii), we can take a sufficiently small neighborhood $U$ of $\bar x$ such that $DF(x)\R^n-\cN=\R^m$ holds for all $x\in U$, which means we may assume $DF(x^k)\R^n-\cN =\R^m$ for all $k\in\n$. We define $K:=\{x\in\R^n:F(x)\in \cN\}$. Applying \cite[Corollary 2.91]{bonnans2013perturbation}, it follows $T_K(x^k) = DF(x^k)^{-1}T_{\mathcal N}(F(x^k)) = DF(x^k)^{-1} \aff(\cN)$, where the second equality is due to $F(x^k) \in \ri(\cN)$, \cite[Section A.5.3]{hiriart2004fundamentals}. Hence, we can infer $h^k \in T_K(x^k)$ and there is $t_{k,i}\downarrow 0$ as $i\to\infty$ such that $F(x^k+t_{k,i}h^k)\in \cN = N_{\pvpd}(\bar\lambda)$. In particular, we may select $t_k\downarrow 0$ with $F(x^k+t_kh^k)\in \cN$. Applying \cite[Example 11.4]{rockafellar2009variational}, we see that $\bar\lambda\in \partial\sigma_\pvpd(F(x^k))$, $\sigma_\pvpd(F(x^k))=\langle \bar\lambda, F(x^k) \rangle$, $\sigma_\pvpd(F(x^k+t_kh^k))=\langle \bar\lambda, F(x^k+t_kh^k) \rangle$ and thus, by \cref{prop4-4}, we have $DF(x^k)^\top\bar\lambda\in \partial\vp(x^k)$. Therefore, by the definition of strict second subderivative, we obtain
\begin{align*}
  &\rd^2_s\vp(\bar x|DF(\bar x)^\top \bar\lambda)(h)\leq {\lim}_{k\rightarrow 0}~\Delta_{t_k}^2  \vp(x^k|DF(x^k)^\top\bar\lambda)(h^k). \\&=\lim_{k\rightarrow 0}\frac{\vp_d(F(x^k+t_kh^k))-\vp_d(F(x^k))-t_k\langle \bar\lambda,DF(x^k)h^k\rangle}{\frac{1}{2}t_k^2} \\
  &=\lim_{k\rightarrow 0}\frac{\langle  \bar\lambda, F(x^k+t_kh^k)\rangle-\langle \bar\lambda,F(x^k)\rangle-t_k\langle \bar\lambda, DF(x^k)h^k\rangle}{\frac{1}{2}t_k^2}=\langle \bar\lambda,D^2F(\bar x)[h,h]\rangle. 
\end{align*}
Moreover, by the convexity of $\vp_d$, it follows $\rd^2_s\vp_d(0|\bar\lambda)(w) \geq 0$ for all $w$, \cite[Proposition 13.20]{rockafellar2009variational}. Consequently, for all $h \in \Rn$, we can infer  
\begin{equation*}
  \begin{aligned}
    \rd^2_s\vp(\bar x|DF(\bar x)^\top \bar \lambda)(h)\geq\rd^2_s\vp_d(0|\bar\lambda)(DF(\bar x)h)+\langle \bar \lambda,D^2F(\bar x)[h,h]\rangle=\langle \bar\lambda,D^2F(\bar x)[h,h]\rangle. 
  \end{aligned}
\end{equation*}
Combining the previous results, it holds that $\rd^2_s\vp(\bar x|DF(\bar x)^\top \bar \lambda)(h) = \langle \bar \lambda,D^2F(\bar x)[h,h]\rangle$ and $\rd^2_s\vp_d(0|\bar\lambda)(DF(\bar x)h) = 0$ for all
$h\in DF(\bar x)^{-1}\aff(N_{\pvpd}(\bar\lambda))$. 
\end{proof}

\subsection{Characterization of Strict Twice Epi-Differentiability}
\label{sec:charc-strict}
In this section, we provide different equivalent characterizations of the strict twice epi-differentiability of strictly decomposable functions. 

Using \cref{thm-strict-chain}, we can show that strict twice epi-differentiability and strict decomposability imply the strict complementarity condition. 
\begin{prop}
\label{prop:strct_full}
    Let $\vp$ be $C^2$-strictly decomposable at $\bar x$ for $\bar\lambda$ with decomposition pair $(\sigma_\pvpd,F)$. If $\vp$ is strictly twice epi-differentiable at $\bar x$ for $DF(\bar x)^\top\bar\lambda$, then $\bar\lambda\in\ri(\pvpd)$ and $\vp$ is $C^2$-fully decomposable at $\bar x$. 
\end{prop}

\begin{proof}
     Assume that $\vp$ is strictly twice epi-differentiable at $\bar x$ for $DF(\bar x)^\top\bar\lambda$. We again set $\cN:=N_\pvpd(\bar\lambda)$. By \cref{thm-strict-chain}, it follows 
     \[ DF(\bar x)^{-1}\aff(\cN)\subseteq \dom(\rd^2_s\vp(\bar x|DF(\bar x)^\top\bar\lambda)). \] 
     When $\vp$ is strictly twice epi-differentiable at $\bar x$ for $DF(\bar x)^\top\bar\lambda$, it holds that
    \[ \rd^2_s\vp(\bar x|DF(\bar x)^\top\bar\lambda)=\rd^2\vp(\bar x|DF(\bar x)^\top\bar\lambda). \]
    By \cref{prop:soc:decomp-subquad} and $\rd\vp(\bar x)(h) = \sigma_\pvpd(DF(\bar x)h)$, we have $\dom(\rd^2\vp(\bar x|DF(\bar x)^\top\bar\lambda)) = \{h\in\R^n:\rd\vp(\bar x)(h)=\langle \bar \lambda, DF(\bar x)h \rangle\} = DF(\bar x)^{-1}\cN$, cf. \cite[Definition 5.1.5]{milzarek2016numerical}. This yields $DF(\bar x)^{-1}\cN=DF(\bar x)^{-1}\aff(\cN)$. We claim this implies that $\cN$ (or, equivalently, $\cN^\circ$) is a linear subspace. Setting $A := DF(\bar x)$ and since $A^{-1}\cN$ is a linear subspace, $A^\top \cN^\circ$ is also a linear subspace. Hence, for all $v\in \cN^\circ$, there is $w\in \cN^\circ$ such that $A^\top v=-A^\top w$, or equivalently, $A^\top(v+w)=0$. 
    Using \cref{nprop4-3}, the condition $\cR(A)-\cN=\R^m$ is equivalent to $\mathrm{ker}(A^\top)\cap \mathcal N^\circ = \{0\}$. Thus, we deduce that $v+w=0$ and $-\cN^\circ \subseteq \cN^\circ$, indicating that $\cN^\circ$ and $\cN$ are linear subspaces. Consequently, by \cite[Proposition 2.2]{lemarechal1996more}, we have $\bar\lambda\in\ri(\pvpd)$, in which case the strict condition \cref{eq:strict-cq} reduces to the nondegeneracy condition \cref{eq:nondeg-cq}. 
\end{proof}

In the following and motivated by \cite[Theorem 4.7]{hang2024chain} and  \cite[Theorem 5.8]{hang2023smoothness}, we provide full connections between strict twice epi-differentiability, the strict complementarity condition, and the continuous differentiability of the proximity operator. 

\begin{thm}
  \label{equi_strict_diff}
  Assume that $\vp$ is $C^2$-strictly decomposable at $\bar x$ for $\bar\lambda$ with decomposition pair $(\vp_d,F) \equiv (\sigma_\pvpd,F)$ and let $\bar v=DF(\bar x)^\top\bar \lambda\in\partial\vp(\bar x)$ with $\bar\lambda\in\partial\vp_d(0) = \pvpd$ be given. Suppose that $\vp$ globally prox-regular at $\bar x$ for $\bar v$ with constant $\rho$ and $\tau\rho<{1}$. Then, the following statements are equivalent:
\begin{enumerate}[label=\textup{\textrm{(\roman*)}},topsep=0pt,itemsep=0ex,partopsep=0ex]
    \item The function $\vp$ is strictly twice epi-differentiable at $\bar x$ for $\bar v$.
    \item There is a neighborhood $U$ of $(\bar x,\bar v)$ such that for all $(x,v)\in U\cap\mathrm{gph}(\partial\vp)$, the function $\vp$ is strictly twice epi-differentiable at $x$ for $v$.
    \item The strict complementarity condition $\bar v\in \ri(\partial\vp(\bar x))$ is satisfied.
    \item We have $\bar\lambda\in\ri(\pvpd)$. 
    \item The proximity operator $\proxs$ is continuously differentiable around $\bar x+\tau\bar v$.  
    \item The proximity operator $\proxs$ is differentiable at $\bar x+\tau\bar v$. 
  \end{enumerate}
\end{thm}

The equivalence between (i) and (ii) implies that the strict twice epi-differentiability of $\vp$ on a neighborhood of $(\bar x,\bar v)$ as in \cite[Theorem 3.9]{hang2024chain} is often unnecessary. Conditions (iii) and (iv) are the strict complementarity conditions for $\vp$ and $\vp_d$, respectively. 

Compared to \cite[Theorem 5.8]{hang2023smoothness}, \cref{equi_strict_diff} allows relaxing the nondegeneracy condition \cref{eq:nondeg-cq} (utilized in \cite{hang2023smoothness}) to the \emph{weaker constraint qualification} \cref{eq:strict-cq}. Furthermore, we only assume $C^2$-strict decomposability of $\vp$ at $\bar x$, while \cite[Theorem 5.8]{hang2023smoothness} requires $C^2$-full decomposability of $\vp$ on a \emph{neighborhood of $\bar x$}.

Besides the results in \cite{hang2024chain,hang2023smoothness}, characterizations of the differentiability of the proximity operator via the strict complementarity condition have been obtained earlier in \cite{DanHarMal06,shapiro2016diff,milzarek2016numerical}. In particular, in \cite[Proposition 3.1]{shapiro2016diff}, Shapiro studied the differentiability of projections onto $C^2$-cone reducible sets under the strict complementarity condition. In \cite[Lemma 5.3.32]{milzarek2016numerical}, an extension of Shapiro's result to $C^2$-fully decomposable functions was established. In \cite[Theorem 28]{DanHarMal06}, Daniilidis et al$.$ further showed that the proximity operator of a prox-bounded, prox-regular, and $C^2$-partly smooth is $C^1$ if the strict complementarity condition is satisfied; see also \cite{HuTiaPanWen23}. 

A main step in the proof is to derive the strict twice epi-differentiability of $\vp$ on a neighborhood of $(\bar x,\bar v)$ from continuous differentiability of the proximity operator. Our primary tool is \cite[Corollary 4.3]{poliquin1996generalized}. We plan to transform the convergence of the Jacobian matrix of the proximity operator to the epi-convergence of the second subderivative. Using \cite[Corollary 4.3]{poliquin1996generalized}, this then allows establishing the strict twice epi-differentiability on a neighborhood of $(\bar x,\bar v)$. We first present a technical preparatory result on the weak convexity of the second subderivative.

\begin{prop}
  \label{prop3-33}
  Suppose that $\vp$ is globally prox-regular at $x$ for $v$ with constant $\rho$ and $\tau\rho<{1}$. Assume further that $\vp$ is properly twice epi-differentiable at $x$ for $v$. Then, $\rd^2\vp(x|v)$ is $\rho$-weakly convex and we have $\rd^2\vp(x|v)+\frac{\rho}{2}\|\cdot\|^2\geq 0$.
\end{prop}
\begin{proof}
  The weak convexity is shown in \cite[Proposition 13.49]{rockafellar2009variational}. Since $\vp$ is properly twice epi-differentiable and $w \mapsto \rd^2\vp(x|v)(w)$ is lsc and positively homogeneous of degree 2, we can further infer $\rd^2\vp(x|v)(0)=0$ and $0\in \partial(\rd^2\vp(x|v)+\frac{\rho}{2}\|\cdot\|^2)(0)$, see \cite[Proposition 13.5]{rockafellar2009variational}. This implies $\rd^2\vp(x|v)+\frac{\rho}{2}\|\cdot\|^2\geq 0$ by convexity.
\end{proof}
\begin{lem}
  \label{lemma2-15}
  Suppose that $\vp$ is globally prox-regular at $x$ for $v$ with constant $\rho$ and $\tau\rho<1$. Let $Q$ be a symmetric matrix satisfying $Q+\rho I\succeq 0$ and $\mathcal R(Q)\subseteq S$ where $S$ is a linear subspace and set $P:=\Pi_S(I+\tau Q)^{-1}\Pi_S$, $R:=\frac{1}{\tau}(I-P)$. Furthermore, let $\gph(\partial \vp) \ni (x^k,v^k) \to (x,v)$ be given and define $z^k=x^k+\tau v^k$. Then, the following conditions are equivalent:
\begin{enumerate}[label=\textup{\textrm{(\roman*)}},topsep=0pt,itemsep=0ex,partopsep=0ex]
      \item The function $\vp$ is twice epi-differentiable at $x^k$ for $v^k$, each $\rd^2\vp(x^k|v^k)$ is generalized quadratic, and we have $\rd^2\vp(x^k|v^k) \elimit \iprod{\cdot}{Q\cdot}+\iota_{S}(\cdot)$ as $k\to \infty$.
      \item The Moreau envelope $\mathrm{env}_{\tau \vp}$ has a Hessian matrix at each $z^k$, and each mapping $e_{\tau}(\frac{1}{2}\rd^2\vp(x^k|v^k))$ is quadratic and converges point-wisely to $\frac{1}{2}  \langle \cdot,  R \cdot \rangle$. 
      \item The proximity operator $\proxs$ is differentiable at $z^k$ with $D\proxs(z^k)\to P$.
  \end{enumerate}
\end{lem}
\begin{proof}
(\rnum{1})\!\,$\iff$\!\,(\rnum{2}): The equivalence of the first properties in (i) and (ii) is shown in \cref{prop_equi_quad_diff}. Noticing $ e_{\tau}(\frac12\iprod{\cdot}{Q\cdot}+\iota_{S}(\cdot))=\frac12\langle \cdot,R\cdot\rangle$, the equivalence between epi-convergence of $\{\rd^2\vp(x^k|v^k)\}_k$ and point-wise convergence of $\{e_{\tau}(\frac{1}{2}\rd^2\vp(x^k|v^k))\}_k$ follows from \cref{prop3-33}, \cref{prop2-5}, and \cite[Exercise 7.8]{rockafellar2009variational}. 

(\rnum{2})\!\,$\iff$\!\,(\rnum{3}): Thanks to \cref{prop_equi_quad_diff}, we only need to verify the convergence properties. Let $\{R_k\}_k$ be given with $e_\tau( \frac{1}{2}\rd^2\vp(x^k|v^k))=\frac{1}{2}\langle \cdot,R_k\cdot \rangle$ for all $k$. Applying \cref{prop-lowerbound-Q} and \cite[Exercise 13.45]{rockafellar2009variational}, we have $\nabla^2e_\tau(\frac{1}{2}\rd^2\vp(x^k|v^k) )=\nabla^2\mathrm{env}_{\tau \vp}(z^k)=R_k=\frac{1}{\tau}(I-D\proxs(z^k))$. Hence, $D\proxs(z^k)\to P$ is equivalent to $R_k\to R$. (Notice again that $\vp$ can be assumed to be globally prox-regular near $(x, v)$ with the same constant $\rho$). It suffices to prove that for quadratic symmetric forms $\langle \cdot, R_k\cdot \rangle$, point-wise convergence to $\langle \cdot, R\cdot \rangle$ is equivalent to $R_k\to R$ for which we refer to \cite{1953973}.
%
%
\end{proof}
\begin{proof}[Proof of \cref{equi_strict_diff}]
The implications ``(ii)\!\,$\implies$\!\,(i)'' and ``(v)\!\,$\implies$\!\,(vi)'' are immediate. ``(iii)\!\,$\iff$\!\,(iv)'' follows from \cite[Theorem 6.6]{rockafellar1970convex}, \cref{prop4-4}, and $\Lambda(\bar x,\bar v) = \{\bar\lambda\}$. The implication ``(i)$\implies$(iv)'' follows from \cref{prop:strct_full}.

Under condition (iii) (or (iv)) and as argued in the proof of \cref{prop:strct_full}, \cref{eq:strict-cq} reduces to the nondegeneracy condition \cref{eq:nondeg-cq} and hence, $\vp$ is $C^2$-fully decomposable at $\bar x$. Moreover, by \cref{fact:structure}, $\vp$ is $C^2$-partly smooth at $\bar x$ (relative to the $C^2$-manifold $\mathcal M = \{x: F(x) \in \lin(N_\pvpd(\bar \lambda))\}$, see \cite{shapiro2003class} or \cite[Section 5]{hare2006functions}).
%
%
   Thus, ``(iii)\!\,$\iff$\!\,(vi)'' follows from \cite[Lemma 5.3.32]{milzarek2016numerical} and $\bar x = \proxs(\bar x+\tau\bar v)$\footnote{The result $\bar x=\proxs(\bar x+\tau \bar v)$ follows from the definition of global prox regularity and the choice $\tau\rho<1$. The implication ``(vi)\!\,$\implies$\!\,(iii)'' can be inferred from the proof of \cite[Lemma 5.3.32]{milzarek2016numerical} even if $\vp$ is only $C^2$-strictly decomposable. Alternatively, \cref{prop_equi_quad_diff} and \cref{prop:soc:decomp-subquad} imply that $DF(\bar x)^{-1}N_\pvpd(\bar\lambda)$ is a subspace. We can then follow the proof of \cref{prop:strct_full}.} and the implication ``(iii)\!\,$\implies$\!\,(v)'' follows from \cite[Proposition 10.12]{drusvyatskiy2014optimality} and \cite[Theorem 3.1]{DavDru22} or \cite[Lemma 4.3]{HuTiaPanWen23}. 
 
  We now prove ``(v)\!\,$\implies$\!\,(ii)''. 
%
%
  Assume that $\proxs$ is $C^1$ on a neighborhood $U$ of $\bar z=\bar x+\tau\bar v$. Then, $\envs$ is $C^2$ on $U$. Let $V$ be a neighborhood of $(\bar x,\bar v)$ such that for all $(x,v)\in V$ it holds that $x+\tau v\in U$. Select any $V\cap\gph(\partial\vp)\ni (x^k,v^k)\to (x,v)\in V$ and set $z^k=x^k+\tau v^k$ and $z=x+\tau v$.
   By \cite[Exercise 13.45]{rockafellar2009variational}, we have $\rd^2(\frac{1}{2}\envs(z^k))=e_{\tau}(\frac{1}{2}\rd^2\vp(x^k|v^k))$, which implies $e_\tau(\frac{1}{2}\rd^2\vp(x^k|v^k))(w)=\frac{1}{2\tau}\langle w,(I-D\proxs(z^k))w\rangle$ for all $w$. Since $\proxs$ is $C^1$ on $U$, this further yields $ e_\tau(\frac{1}{2}\rd^2\vp(x^k|v^k))\overset{p}{\rightarrow}e_\tau(\frac{1}{2}\rd^2\vp(x|v))$. By \cref{prop_equi_quad_diff}, $\vp$ is properly twice epi-differentiable at $x^k$ for $v^k$ (at $x$ for $v$) for all $k$. Thus, using \cref{prop3-33} and \cref{prop2-5}, we can infer $\rd^2\vp(x^k|v^k)\overset{e}{\rightarrow }\rd^2\vp(x|v)$. According to \cite[Corollary 4.3]{poliquin1996generalized}, this implies that $\vp$ is strictly twice epi-differentiable at $x$ for $v$ (where $(x,v)\in V$).
\end{proof}

As an application, we can fully characterize the notion of strict saddle points proposed in \cite[Definition 2.7]{DavDru22} complementing the discussion in \cite[Lemma 1.4.8]{DavDruJia21}. 
\begin{thm} \label{thm:saddle}
  Suppose that $\vp$ is $C^2$-strictly decomposable at $\bar x$ with decomposition pair $(\sigma_\pvpd,F)$. Then, $\bar x$ is a strict saddle point of $\vp$ iff $0\in\ri(\partial\vp(\bar x))$ and there is a vector $h\in \mathcal C{(\bar x)}:=\{h\in\R^n:\rd\vp(\bar x)(h)=0\}$ such that $\rd^2\vp(\bar x|0)(h)<0$.
\end{thm}
\begin{proof}
  Combining our previous discussion and \cite[Proposition 10.12]{drusvyatskiy2014optimality}, $\vp$ has a $C^2$-active manifold $\cM$ at $\bar x$ for $0$ iff $0\in\ri(\partial\vp(\bar x))$. Since such an active manifold is locally unique, \cite[Propositions 2.4 and 10.10]{drusvyatskiy2014optimality}, it must coincide with the manifold $\cM=\{x\in\R^n:F(x)\in L\}$ locally, where $L=\lin(N_{\pvpd}(\bar\lambda))$. 
  Moreover, due to the nondegeneracy condition, it holds that $T_{\cM}(\bar\lambda)=DF(\bar x)^{-1}L$ and locally, we obtain $\vp_\cM(x) := (\vp+\iota_\cM)(x) = \vp(\bar x) + \iprod{\bar \lambda}{F(x)}+\iota_{\cM}(x)$. Thus, $\bar x$ is a strict saddle point iff, in addition, $\rd^2\vp_\cM(\bar x)(h)=\langle \bar\lambda, D^2F(\bar x)[h,h] \rangle<0$ for some $h\in DF(\bar x)^{-1}L$. According to \cref{prop:soc:decomp-subquad} and \cref{equi_strict_diff}, we have $0=DF(\bar x)^\top\bar\lambda$ with $\bar\lambda\in\ri(\pvpd)$ and $\rd^2\vp(\bar x|0)=\langle \bar\lambda, D^2F(\bar x)[\cdot,\cdot] \rangle +\iota_{DF(\bar x)^{-1}L}$. Since $N_{\pvpd}(\bar\lambda)$ is a subspace (see \cite[Proposition 2.2]{lemarechal1996more}) and as in the proof of \cref{prop:strct_full},
  it follows $DF(\bar x)^{-1}L = \{h: DF(\bar x)h \in N_{\pvpd}(\bar\lambda)\} = \{h: \sigma_{\pvpd}(DF(\bar x)h) = 0\} = \mathcal C(\bar x)$ which concludes the proof.
\end{proof}


\subsection{Uniform Decomposability and the Strict Second Subderivative}
\label{sec:uniform-d}


\cref{thm-strict-chain} links the strict second subderivative of $\vp$ to the one of $\vp_d$. However, computational results that allow characterizing the strict second subderivative of the support function $\vp_d \equiv \sigma_{\mathcal Q}$ do not seem to be available in the literature. We will tackle the computation of $\rd^2_s\sigma_\pvpd(0|\bar\lambda)$ via a geometric perspective. We first present a general result for $\rd^2_s\sigma_\pvpd(0|\bar\lambda)$ without any assumption on $\pvpd$, which shows that $\rd^2_s\sigma_\pvpd(0|\bar\lambda)$ must be an indicator function of a closed cone. 


 \begin{prop} Let $\pvpd\subseteq \R^m$ be a nonempty, closed, and convex set. For all $\bar\lambda\in \pvpd$ and $w\in \R^m$, we have:
   \[       \rd^2_s\sigma_{\mathcal Q}(0|\bar\lambda)(w)=\iota_{\cS}(w),                   \]
   where $\cS \subseteq \Rm$ is a closed cone.
 \end{prop}
 \begin{proof}
   Due to the convexity of $\sigma_{\mathcal Q}$, it follows $\rd^2_s\sigma_{\mathcal Q}(0|\bar\lambda)(w)\geq 0$ (cf. the proof of \cite[Proposition 13.20(a)]{rockafellar2009variational}). Next, let us assume  $\rd^2_s\sigma_{\mathcal Q}(0|\bar\lambda)(w)=\alpha$ for some $\alpha\geq 0$. By definition, there exist $x^k\to 0$, $t_k\downarrow 0$, $w^k\to w$, and $\lambda^k\in\partial\sigma_{\mathcal Q}(x^k)$ with $\lambda^k\to\bar\lambda$ such that $\alpha ={\lim}_{k\to\infty} \Delta_{t_k}^2  \;\! \sigma_{\mathcal Q}(x^k|\lambda^k)(w^k)$. 
   We now set $\tilde x^k:=\beta x^k$, $\tilde t_k:=\beta t_k$ for some $\beta>0$. By the positive homogeneity of $\sigma_{\mathcal Q}$, we have $\lambda^k\in\partial \sigma_{\mathcal Q}(\tilde x^k)$ (cf. \cite[Example 11.4]{rockafellar2009variational}) and it holds that:
   \[  \frac{\alpha}{\beta} = \lim_{k\to\infty} \frac{\sigma_{\mathcal Q}(\tilde x^k+\tilde t_kw^k)-\sigma_{\mathcal Q}(\tilde x^k)-\tilde t_k\langle \lambda^k,w^k\rangle}{\frac{1}{2}\tilde t_k^2} \geq     \rd^2_s\sigma_{\mathcal Q}(0|\bar\lambda)(w)=\alpha.                              \]
   This shows $\alpha\geq \alpha\beta$ for all $\beta>0$, or equivalently, $\alpha\in\{0,\infty\}$. Consequently, we see that $\rd^2_s\sigma_{\mathcal Q}(0|\bar\lambda)=\iota_{\cS}$ for some set $\cS$. Since $\rd^2_s\sigma_{\mathcal Q}(0|\bar\lambda)$ is lsc and positively homogeneous with degree 2 (see the discussion before \cite[Lemma 3.2]{hang2024chain}), $\cS$ is a closed cone.
 \end{proof}

\subsubsection{Uniform Decomposability}
\label{sec:uni_dec}

As mentioned, the analysis of the strict second subderivative requires to suitably control the second order behavior of $\vp$ and $\vp_d$ on a neighborhood rather than just at a single point. Not surprisingly, this also depends on the geometric features of the support set $\pvpd$. In the following, we introduce a novel geometric property for the set $\pvpd$ that will allow us to derive explicit representations of the strict second subderivative $\rd^2_s \vp_d \equiv \rd^2_s \sigma_{\mathcal Q}$. 
\begin{defn}
  \label{def-usotp}
  We say a closed convex set $\pvpd$ has the uniform second-order tangent path ($\usotp$-)property at $\bar\lambda \in \pvpd$, if there are uniform constants $\delta>0$ and $M>0$ and a neighborhood $V$ of $\bar\lambda$ such that for all $\lambda\in \pvpd \cap V$ it holds that
\begin{enumerate}[label=\textup{\textrm{(\roman*)}},topsep=0pt,itemsep=0ex,partopsep=0ex]
    \item We can select a linear subspace $S(\lambda)\subseteq \lin(T_{\pvpd}(\lambda))$ such that $S(\lambda)\to \lin(T_{\pvpd}(\bar\lambda))$ as $\lambda\to\bar\lambda$, 
    \item For all $v\in S(\lambda)$ with $\|v\|=1$, there exists a path $\xi:[0,\delta]\to \pvpd$ such that $\xi(t)=\lambda+tv+\frac{1}{2}t^2r(t)$ with $\|r(t)\|\leq M$ for all $t\in [0,\delta]$. 
  \end{enumerate} 
  If this holds for all $\bar\lambda\in \pvpd$, then we say $\pvpd$ has the $\usotp$-property.
\end{defn}

The first requirement in \cref{def-usotp} is equivalent to the inner semicontinuity (isc) of the set-valued mapping $\lambda \mapsto \lin(T_{\pvpd}(\lambda))$ at $\bar\lambda$ relative to $\pvpd$. Although the set-valued mapping $T_{\pvpd}(\lambda)$ is isc everywhere relative to $\pvpd$ (due to the convexity of $\pvpd$), the same property does not hold for $\lin(T_{\pvpd}(\lambda))$ in general. Moreover, condition (ii) in \cref{def-usotp} necessarily requires the outer second order tangent set $T_{\pvpd}^2(\lambda,v) := \{w: \exists~t_k \downarrow 0 \, \text{such that} \, \dist(\lambda+t_kv+\frac12t_k^2 w,\pvpd) = o(t_k^2)\}$ to be nonempty for all $\lambda \in V$ and $v\in S(\lambda)$ with $\|v\|=1$.

\begin{defn}
\label{defn3-21}
  We say that $\vp$ is $C^2$-strictly (fully) $u$-decomposable at $\bar x$ for $\bar\lambda$, if $\vp$ is $C^2$-strictly (fully) decomposable at $\bar x$ for $\bar\lambda$ and $\pvpd$ has the $\usotp$-property at $\bar \lambda$.
\end{defn}

The $\usotp$-property holds for a large variety of sets including the important class of $C^2$-cone reducible sets, see \cite[Definition 3.135]{bonnans2013perturbation}.
\begin{defn} \label{def:cone-red}
  A closed set $\mathcal S \subseteq \Rn$ is said to be $C^2$-pointed with respect to the pair $(G,\mathcal K)$, at a point $\bar s \in \mathcal S$, if 
  \begin{itemize}
  \item $\mathcal K \subseteq \Rm$ is a closed, convex set and we have $\lin(T_{\mathcal K}(G(\bar s)))=\{0\}$. 
  \item There is a neighborhood $U$ of $\bar s$ such that $G : U \to \Rm$ is $C^2$, $DG(\bar s)$ is onto, and the set
  $\mathcal S$ can be represented as $\mathcal S \cap U = \{s\in U:G(s)\in \mathcal K\}$.
\end{itemize}  
Here, we identify $\R^0$ with $\{0\}$ to include the case where $\bar s$ is in the relative interior of $\cS$.
If, in addition, $\mathcal K -G(\bar s)$ is a closed, convex, pointed cone, then $\mathcal S$ is called $C^2$-cone reducible to $\mathcal K$ at $\bar s$.
\end{defn}

In \cref{def:cone-red} and similar to \cref{def:decomp}, we can typically assume $G(\bar s) = 0$  (without loss of generality). In this case, the condition $\lin(T_{\mathcal K}(G(\bar s))) = \{0\}$ is equivalent to $\lin(\mathcal K) = \{0\}$, i.e., the set $\mathcal K$ is required to be pointed. We now verify that $C^2$-pointed (and thus $C^2$-cone reducible) sets satisfy the $\usotp$-property. 


\begin{prop}
  \label{uniform_soct}
 Let $\mathcal S$ be 
 $C^2$-pointed at $\bar s \in \mathcal S$. Then, $\mathcal S$ has the $\usotp$-property at $\bar s$. 
\end{prop}
\begin{proof} Let the reduction pair $(G,\mathcal K)$ and the neighborhood $U$ (of $\bar s$) be given as in \cref{def:cone-red}.
The surjectivity of $DG(\bar\lambda)$ implies that Robinson's constraint qualification $0 \in \mathrm{int}\{G(\bar s) + DG(\bar s)\Rn - \mathcal K\}$ holds at $\bar s$. Hence, \cite[Theorem 2.87 and Corollary 2.91]{bonnans2013perturbation} are applicable and (adjusting $U$ if necessary) there is $\kappa_1>0$ such that
  \begin{align}
      \label{eq_lipschitz_U_K}
\dist(s, \mathcal S)\leq \kappa_1 \dist(G(s),\mathcal K) \quad \forall~s \in U.  
  \end{align}
  We also have $T_{\mathcal S}(s)=DG(s)^{-1}(T_{\mathcal K}(G(s)))$ and $\lin(T_{\mathcal S}(s))=DG(s)^{-1}(\lin(T_{\mathcal K}(G(s))))$  which implies $\lin(T_{\mathcal S}(\bar s)) = \ker(DG(\bar s))$ and $\ker(DG(s))\subseteq \lin(T_{\mathcal S}(s))$ for all $s \in U$.
  %
Therefore, we may define $S(s) :=\ker(DG(s))$ in \cref{def-usotp}. The convergence $S(s) \to \ker(DG(\bar s))$ then follows from \cite[Theorem 4.32 (b)]{rockafellar2009variational}.
Let $V$ be a neighborhood of $\bar s$ and let $\veps >0$ be given such that $\B_\veps(s)\subseteq U$ for all $s\in V$. We further assume that $DG(s)$ has full row rank for all $s \in V$. Let $s \in V$ and $v\in S(s)$ with $\|v\| = 1$ be arbitrary. We set $q :=-DG(s)^\dagger D^2G(s)[v,v]$ where $DG(s)^\dagger = DG(s)^\top(DG(s)DG(s)^\top)^{-1}$ is the Moore-Penrose inverse of $DG(s)$. Then, using Taylor expansion, we may write
  \begin{align*}  G(s+tv+\tfrac{1}{2}t^2q) &= G(s) + \tfrac12 t^2 e_{s,v}(t). \end{align*}
Notice, due to the continuity of $D^2G$ and $DG^\dagger$, the error term $e_{s,v}$ satisfies $e_{s,v}(t) = o(1)$ uniformly for all $s$ and $v$. In particular, there is $C > 0$ such that $\|e_{s,v}(t)\| \leq C$ for all $s \in V$, $\|v\| = 1$, and $t \in [0,\veps]$. In addition, there exists $\delta \in (0,\veps]$ such that 
%
  \[   s + tv + \tfrac12 t^2 q = s+tv-\tfrac{1}{2}t^2DG(s)^\dagger D^2G(s)[v,v] \in \B_{\veps}(s) \quad \forall~s \in V, \;\; \|v\| = 1, \;\; t \in [0,\delta].             \]
 Thus, applying \cref{eq_lipschitz_U_K}, we can infer
  \begin{align*}
   \dist(s+tv+\tfrac{1}{2}t^2q,\mathcal S)&\leq \kappa_1\dist(G(s+tv+\tfrac{1}{2}t^2q),\mathcal K)\\ & =\kappa_1\dist(G(s)+\tfrac{1}{2}t^2e_{s,v}(t),\mathcal K) \leq \frac{\kappa_1}{2}t^2\|e_{s,v}(t)\|\leq \frac{C\kappa_1}{2}t^2.
  \end{align*}
  for all $s \in \cS \cap V$, $v \in S(s)$, $\|v\| = 1$, $t\in [0,\delta]$. 
 Finally, let us choose the path $\xi(t)$ via $\xi(t)\in\Pi_{\mathcal S}(s+tv+\frac{1}{2}t^2q)$ and define $ r(t):=2({\xi(t)-(s+tv)})/t^2$. Then, it holds that 
  \[ \xi(t)=s+tv+\tfrac{1}{2}t^2 r(t), \quad \|r(t)\|\leq  \frac{\|\xi(t)-(s+tv+\frac{1}{2}t^2q)\|+\frac{1}{2}t^2\|q\|}{\frac{1}{2}t^2}\leq  C\kappa_1 + M_1,          \]
  where we assume $\|q\| = \|DG(s)^{\dagger}D^2G(s)[v,v]\|\leq M_1$ for all {$s \in V$}, $\|v\| = 1$. (Existence of such $M_1$ again follows from continuity). It suffices to take $M:=C\kappa_1+M_1$.
\end{proof}

$C^2$-cone reducible sets allow covering many practical examples, such as the second-order cone, the set of positive semidefinite matrices, or polyhedral sets.


\subsubsection{Main Result and the Strict Second Subderivative} 
We now present one of the main results of this work. In particular, if $\pvpd$ has the $\usotp$-property, we show that the strict second subderivative of $\vp$ at $\bar x$ is the sum of the quadratic form related to $F$ and $\bar \lambda$ and the indicator function of the affine hull of the critical cone. 
\begin{thm}
  \label{thm_pdecom_strict_twice}
  Let $\vp$ be $C^2$-strictly $u$-decomposable at $\bar x$ for $\bar\lambda$ with decomposition pair $(\vp_d,F) \equiv (\sigma_\pvpd,F)$. Then, it holds that:
  \begin{align*}  \rd_s^2\vp(\bar x|DF(\bar x)^\top\bar\lambda)(h) = \langle \bar\lambda,D^2F(\bar x)[h,h]\rangle + \iota_{\aff(DF(\bar x)^{-1}N_{\pvpd}(\bar\lambda))}(h) \quad \forall~h \in \Rn. \end{align*}
%
%
\end{thm}

Setting $\bar v := DF(\bar x)^\top\bar\lambda$ and as discussed before, the set $DF(\bar x)^{-1}N_{\pvpd}(\bar\lambda)$ coincides with the critical cone $\mathcal C(\bar x,\bar v)$ defined in \cref{prop:soc:decomp-subquad}. We further note that 
the formula in \cref{thm_pdecom_strict_twice} is not true in general, as illustrated by the following example.
\begin{example}
  Let us consider 
  \[ C=\{ x \in \R^2: x_2\geq \tfrac{2}{3}|x_1|^\frac32\},\quad F=I,\quad \vp(x)=\vp_d(x)=\sigma_C(x), \]
  and $\bar x=0$, $\bar\lambda=0$. Then, for $x_1,x_2 \neq 0$ and $x := (|x_1|,-|x_2|)^\top$, we can compute
  \[ \vp(x)=\sup_{y\in C}~|x_1|y_1 - |x_2|y_2=\sup_{t\geq 0}~t|x_1|-\frac{2}{3}t^{\frac32}|x_2|=\frac{1}{3}\frac{|x_1|^{3}}{x_2^{2}}. \]      
  Clearly, $\vp$ is twice differentiable at $x$, and we have
  \[     \nabla \vp(x) = \frac{|x_1|}{x_2^2}\begin{pmatrix} x_1 \\ -\frac23 \frac{x_1^2}{x_2} \end{pmatrix} \quad \text{and} \quad   \nabla^2\vp(x)=\frac{2|x_1|}{x_2^2}\begin{pmatrix}
     1  & -\frac{x_1}{x_2}  \\
     -\frac{x_1}{x_2}  & \frac{x_1^2}{x_2^2}
  \end{pmatrix}.                            \]
  Defining $x^k=(\frac{1}{k^3},-\frac{1}{k})^\top$ and $\lambda^k=(\frac{1}{k^{4}},\frac{2}{3k^6})^\top$, it follows $x^k, \lambda^k \to 0$, $\lambda^k = \nabla\vp(x^k)$, and $\nabla^2\vp(x^k) \to 0$.
   Moreover, by \cite[Example 13.8]{rockafellar2009variational}, it holds that $\rd^2\vp(x^k|\lambda^k)(w)=\langle w,\nabla^2\vp(x^k)w\rangle$. Therefore, according to \cite[Lemma 3.2]{hang2024chain}, we obtain:
   \[  0\leq \rd^2_s\vp(\bar x|\bar\lambda)(w)\leq {\lim}_{k\to\infty}\,\langle w,\nabla^2\vp(x^k) w\rangle=0 \quad \forall~w \in \Rn,             \]
   which implies $\rd^2_s\vp(\bar x|\bar\lambda)=0$. However, due to $\aff(N_{C}(\bar\lambda))=\aff(\{(0,t):t\leq 0\}) \neq \R^2$, the conclusion in \cref{thm_pdecom_strict_twice} does not hold. In fact, setting $v := (1,0)^\top$ and following \cite[Example 3.29]{bonnans2013perturbation}, we have $v \in \lin(T_C(\bar\lambda))$ and $T_{C}^2(\bar\lambda,v)=\emptyset$. Consequently, the set $C$ does not have the $\usotp$-property at $\bar\lambda$.
\end{example}

Our first goal is to control $\rd^2_s\sigma_\pvpd(0|\bar\lambda)$ along directions $h \notin \aff(N_{\pvpd}(\bar\lambda))$. We start with an auxiliary result.
\begin{lem}
  \label{lem-affhull}
  Let $\cN\subseteq \R^m$ be a closed convex cone and let $A:\R^n\to\R^m$ be a linear mapping. Suppose that the constraint qualification $\cR(A)-\cN=\R^m$ is satisfied. Then, it holds that $A^{-1}\aff(\cN)=\aff(A^{-1}\cN)$.
\end{lem}
\begin{proof}
  Since $\mathcal N$, $A^{-1}\mathcal N$, $A^{-1}\aff(\mathcal N)$, and $\aff(A^{-1}\mathcal N)$ are all closed convex cones, the condition $\aff(A^{-1}\cN)=A^{-1}\aff(\cN)$ is equivalent to
  \[ \lin(A^\top \mathcal N^\circ) = [\aff(A^{-1}\mathcal N)]^\circ = [A^{-1}\aff(\mathcal N)]^\circ = A^\top\lin(\mathcal N^\circ) \] 
  where we used \cite[Corollary 6.33, Proposition 6.4, 6.34, and 6.36]{bauschke2017convex}. We clearly have $A^\top\lin(\mathcal N^\circ) \subseteq \lin(A^\top \mathcal N^\circ)$. By \cref{nprop4-3}, the condition $\cR(A)-\cN=\R^m$ is equivalent to $\mathrm{ker}(A^\top)\cap \mathcal N^\circ = \{0\}$ which allows proving $ \lin(A^\top \mathcal N^\circ) \subseteq A^\top\lin(\mathcal N^\circ)$. 
\end{proof}
\begin{lem}
  \label{lemma-path2ineq}
  Let $\pvpd\subseteq\R^m$ be a closed, convex set and let $t > 0$, $w\in \R^n\backslash\{0\}$, $x\in\dom(\sigma_\pvpd)$, and $\lambda\in \partial \sigma_\pvpd(x)$ be given. Assume there are constants $\delta, M>0$ and a vector $v\in \R^m$, $\|v\|=1$, such that:
\begin{enumerate}[label=\textup{\textrm{(\roman*)}},topsep=0pt,itemsep=0ex,partopsep=0ex]
    \item There exists a path $\xi:[0,\delta]\to\pvpd$, $\xi(s)=\lambda+sv+\frac{s^2}{2}r(s)$ with $\|r(s)\|\leq M$ for all $s\in [0,\delta]$.
    \item It holds that $\langle v,x\rangle=0$ and $\langle v, w\rangle>0$.
  \end{enumerate}
  Then, we have $ \Delta_{t}^2\sigma_\pvpd(x|\lambda)(w)      
     \geq \langle v,w\rangle \cdot \min\{\frac{\delta}{t},\frac{\langle v,w\rangle}{M\|x+tw\|} \}$. 
\end{lem}
\begin{proof}
The condition $\lambda \in \partial\sigma_\pvpd(x)$ implies $\sigma_\pvpd(x)=\langle \lambda, x\rangle$ and we obtain
  \begin{equation}
    \label{eq_vp_minus_vpf}
      \sigma_\pvpd(x+tw)-\sigma_\pvpd(x)-t\langle \lambda, w\rangle = {\sup}_{z\in\pvpd}~\langle z-\lambda,x+tw\rangle 
  \end{equation}
  Setting $g(s):=\langle sv+\frac{s^2}{2}r(s),x+tw\rangle$, it further follows $g(s)\geq st\langle v,w\rangle-\frac{Ms^2}{2}\|x+tw\|$.
  We now define $s^*:=\min\{\delta, \frac{t\langle v,w\rangle}{M\|x+tw\|}\}$ which is positive by condition (ii). If $t\langle v,w\rangle\leq \delta M\|x+tw\|$, then we have:
  \[  g(s^*)\geq {t^2|\langle v,w\rangle|^2} \, / \, {(2M\|x+tw\|)} = {s^*t} \iprod{v}{w} \, / \, 2.      \]
  Otherwise, it holds that $g(s^*)=g(\delta)\geq \delta t\langle v,w\rangle -\frac{\delta^2}{2}M\|x+tw\|\geq \frac{\delta t}{2}\langle v,w\rangle$. 
 Overall, this establishes $g(s^*)\geq \frac{t}{2}\langle v,w\rangle s^*$.
 Using \cref{eq_vp_minus_vpf}, we can then conclude $\Delta_{t}^2\sigma_\pvpd(x|\lambda)(w) \geq \frac{2}{t^2}\sup_{s\in [0,\delta]} \langle sv+\frac{s^2}{2}r(s),x+tw\rangle\geq \frac{2}{t^2} g(s^*)$.
\end{proof}


\begin{proof}[Proof of \cref{thm_pdecom_strict_twice}]  Due to \cref{thm-strict-chain}, we have 
\[ \rd^2_s\vp(\bar x|DF(\bar x)^\top\bar\lambda)(w)=\langle \bar\lambda,D^2F(\bar x)[w,w]\rangle \quad \forall~w\in DF(\bar x)^{-1}\aff(N_{\pvpd}(\bar\lambda)). \]
Thus, in view of \cref{thm-strict-chain} and noticing $\aff(DF(\bar x)^{-1}N_{\pvpd}(\bar\lambda)) = DF(\bar x)^{-1}\aff(N_{\pvpd}(\bar\lambda))$ (by \cref{lem-affhull}), we only need to show 
\[ \rd^2_s\sigma_\pvpd(0|\bar\lambda)(w)=\infty \quad \forall~w\notin \aff(\cN), \quad \cN := N_{\pvpd}(\bar\lambda). \]
%
Let us fix $w\notin \aff(\mathcal N)$. For $v:=\Pi_{\lin(T_{\pvpd}(\bar\lambda))}(w)$, it then follows $\|v\|=\epsilon>0$ where we used $[\aff(\mathcal N)]^\circ=\lin(T_{\pvpd}(\bar\lambda))$. By the definition of the strict second subderivative, we may select $x^k\to 0$, $\partial\sigma_\pvpd(x^k)\ni \lambda^k\to \bar\lambda$, $w^k\to w$, and $t_k\downarrow 0$ such that
  \[  \rd^2_{s}\sigma_\pvpd(0|\bar\lambda)(w)={\lim}_{k\to\infty}~\Delta_{t_k}^2\sigma_\pvpd(x^k|\lambda^k)(w^k).       \] 
%
Let $S(\lambda)$ denote the linear subspace in \cref{def-usotp} and define $P(\lambda):=\Pi_{S(\lambda)}$ and $v^k:=P(\lambda^k)w^k$. By continuity, we have $v^k \to v$ and we may assume $\|v^k\|\geq \frac{\epsilon}{2}$ for all $k$. By the $\usotp$-property and setting $u^k={v^k}/{\|v^k\|}$, there are constants $\delta, M>0$ such that for every $k$, there is a path $\xi^k:[0,\delta]\to\pvpd$ with $\xi^k(s)=\lambda^k+su^k+\frac{s^2}{2}r^k(s)$ and $\|r^k(s)\|\leq M$ for all $s\in [0,\delta]$. Moreover, we have $\iprod{x^k}{u^k} = 0$ (due to $x^k \in N_\pvpd(\lambda^k)$ and $S(\lambda^k)\subseteq \lin(T_\pvpd(\lambda^k))$) and $\langle w^k, u^k\rangle=\|v^k\|\geq \frac{\epsilon}{2}$. Thus, the conditions in \cref{lemma-path2ineq} are satisfied for all $k\in \n$ and we can infer
  \begin{equation*}
    \begin{aligned}
      \Delta_{t_k}^2\sigma_\pvpd(x^k|\lambda^k)(w^k)\geq\langle u^k,w^k\rangle\min\Big\{\tfrac{\delta}{t_k},\tfrac{\langle u^k,w^k\rangle}{M\|x^k+t_kw^k\|}\Big\}\geq \tfrac{\epsilon}{2}\min\Big\{\tfrac{\delta}{t_k},\tfrac{\epsilon}{2M\|x^k+t_kw^k\|}\Big\}.
   \end{aligned}
  \end{equation*}
Noticing $\|x^k+t_kw^k\|\to 0$, this yields $\Delta_{t_k}^2\sigma_\pvpd(x^k|\lambda^k)(w^k) \to \infty$ as $k\to\infty$.
\end{proof}

As an application, we are able to fully characterize the strong metric regularity of $\partial\vp$ at $(\bar x,DF(\bar x)^\top\bar\lambda)$, if $\vp$ is $C^2$-strictly $u$-decomposable at $\bar x$ for $\bar\lambda$.
\begin{thm}
  Let $\vp$ be $C^2$-strictly $u$-decomposable at $\bar x$ for $\bar\lambda$ with decomposition pair $(\sigma_\pvpd,F)$ and set $\bar v=DF(\bar x)^\top\bar\lambda$. Then, the following statements are equivalent:
\begin{enumerate}[label=\textup{\textrm{(\roman*)}},topsep=2pt,itemsep=0ex,partopsep=0ex]
    \item We have $\rd^2_s\vp(\bar x|DF(\bar x)^\top\bar\lambda)(h) > 0$ for all $h \in \Rn \backslash \{0\}$, i.e., 
    \[  \langle \bar\lambda, D^2F(\bar x)[h,h] \rangle>0 \quad \forall~h\in \aff(\mathcal C(\bar x,\bar v)) \backslash \{0\}.   \]
    \item There are neighborhoods $U$ of $\bar x$ and $V$ of $\bar v$ and $\sigma>0$ such that 
    \[   \vp(x')\geq\vp(x)+\langle \bar v,x'-x \rangle+\frac{\sigma}{2}\|x-x'\|^2,\;\forall~(x,v)\in\gph(\partial\vp) \cap(U\times V),\;\forall~x'\in U.        \]
    \item The mapping $\partial\vp$ is strongly metrically regular at $\bar x$ for $\bar v$ and $\bar x$ is a local minimizer of the mapping $\vp-\langle \bar v,\cdot \rangle$.
  \end{enumerate}
\end{thm}
\begin{proof}
  This follows from \cref{thm_pdecom_strict_twice}, \cite[Proposition 3.3]{hang2024chain} and \cite[Theorem 3.7]{DruMorNhg14}. Notice that $\vp$ is prox-regular and subdifferentially continuous at $(\bar x,\bar v)$, which allows us to replace the neighborhood $U_\epsilon$ by $U$ in \cite[Proposition 3.3 (b)]{hang2024chain}. 
\end{proof}

\section{Examples and Applications}
\label{sec:exam}
In this section, we demonstrate the broad applicability of our framework. In particular, we prove that several important examples and applications have a uniformly decomposable structure. 

\subsection{Polyhedral Composite Functions} In \cite{hang2024chain}, Hang and Sarabi study composite functions of the form $\vp(x) = g(\Phi(x))$, $x \in U$, where $g : \Rm \to \Rex$ is a polyhedral function, $\Phi : \Rn \to \Rm$ is $C^2$, and $U$ is a neighborhood of $\bar x \in \dom(\vp)$. In \cite[Example 5.1(b)]{hang2023smoothness}, it is shown that such $\vp$ is $C^2$-fully decomposable at $\bar x$ under the constraint qualification \cite[Equation (2.6)]{hang2024chain}. In addition, the support set $\pvpd$ associated with the corresponding decomposition pair $(\sigma_\pvpd,F)$ is polyhedral. Since polyhedral sets are $C^2$-cone reducible everywhere, the models in \cite{hang2024chain} are $C^2$-fully $u$-decomposable.

\subsection{Cone Reducible Constraints} Let the set $\mathcal S \subseteq \Rn$ be $C^2$-cone reducible at $\bar x \in \mathcal S$  and let $(G,\mathcal K)$ be the associated reduction pair. Then, the function 
\[ \vp(x) = \iota_{\mathcal S}(x) \]
is clearly $C^2$-fully decomposable with $\vp_d = \iota_{\mathcal K} = \sigma_{\mathcal K^\circ}$ and $F = G$. Furthermore, if $\mathcal K^\circ$ itself is $C^2$-cone reducible at every point in $\mathcal K^\circ$, then $\vp$ is $C^2$-fully $u$-decomposable at $\bar x$. Since the cones $\mathcal K$ and $\mathcal K^\circ$ typically have a similar (often much simpler) geometric structure as the base set $\mathcal S$, uniform decomposability of cone reducible constraints occurs quite naturally in many applications. For instance, let us consider the set $\mathbb S_+^n$ of symmetric, positive semidefinite $n \times n$ matrices. The set $\mathbb S_+^n$ is $C^\infty$-cone reducible at every point $\bar X \in \mathbb S_+^n$,  \cite[Example 3.140]{bonnans2013perturbation}. Specifically, if $\bar X$ has rank $r < n$, the cone $\mathcal K$ of the associated reduction pair is given by $\mathcal K = \mathbb S_+^{n-r}$ and we have $\mathcal K^\circ = -\mathbb S_+^{n-r}$. Thus, $\mathcal K^\circ$ is cone reducible everywhere implying full $u$-decomposability of $\iota_{\mathbb S_+^n}$ at all $\bar X \in \mathbb S_+^n$. 

\begin{remark}
\label{remark4-1}
In the definition of $C^2$-cone reducible sets (\cref{def:cone-red}), the cone $\cK$ is required to be pointed. However, this is typically not needed as $G$ can be composed with a linear mapping projecting onto the set $\lin(\cK)^\perp$ which allows ensuring pointedness. Hence, when verifying $C^2$-cone reducibility, we do not explicitly require the cone $\cK$ to be pointed. This is also discussed in \cite[Definition 3.135]{bonnans2013perturbation}.
\end{remark}

We continue with two more practical examples of cone reducible constraints.

\subsubsection{Slices of the Second-order Cone}  

We consider slices of the second-order cone \cite{wang2011strong}:
%
    \[ \cS := \cH\cap\Lambda_+^n, \quad \cH :=\{y\in \R^{n+1}:Ay=b\}, \quad \Lambda^n_+ :=\{x\in\R^{n+1}: \|x_{2:n+1}\| \leq x_1\}, \]
    where $A \in \R^{m \times (n+1)}$ and $b \in \R^m$ are given and $x_{2:n+1} := (x_2,\dots,x_{n+1})^\top$. We further assume $\cS\neq\emptyset$. 
    
    We first discuss the case $ \cH \cap \ri(\Lambda_+^n)=\emptyset$. In this case, it holds that $\cS\subseteq\rbd(\Lambda^n_+) = \Lambda^n_+ \backslash \ri(\Lambda^n_+)$. Let $y\in \ri(\cS)$. If $y=0$, we must have $\cS=\{0\}$ since $0$ is an extreme point of $\Lambda_+^n$. Otherwise, we consider the extreme ray $L(y)$ generated by $y$. According to \cite{malik2005q}, this ray $L(y)$ is a face of $\Lambda^n_+$. By \cite[Theorem 18.1]{rockafellar1970convex}, we can then infer $\cS\subseteq L(y)$ which means $\dim(\cS)\leq 1$. In this case, it is easy to verify that $\vp=\iota_\cS$ is $C^2$-fully decomposable at every point in $\cS$ with the support set being $C^2$-cone reducible everywhere. (This support set  is actually a half space or a linear subspace). Therefore, in the following, we assume $\ri(\Lambda_+^n)\cap \cH\neq \emptyset$.
    
    Take any $\bar x\in \cS$. If $\bar x=0$, then $\cH$ can be written as $\cH=\{x\in \R^{n+1}:Ax=0\}$. According to the proof in \cite[Theorem 3.2]{wang2011strong}, there is an orthogonal matrix $P\in \R^{(n+1)\times(n+1)}$ such that $P\cS$ is one of the following: the singleton $\{0\}$, a ray, or an axis-weighted second-order cone $\Lambda^{p,\epsilon}_+\times\{0\}^{n-p}$, where $\Lambda^{p,\epsilon}_+:=\{y \in\R^{p+1}: \|y_{2:p+1}\| \leq \epsilon y_1\}$. As the first two cases ($\dim(\cS)\leq 1$) have already been discussed, we may assume $P\cS = \Lambda^{p,\epsilon}_+\times\{0\}^{n-p}$ for some $p$ and $\epsilon$. Let $D\in\R^{(n+1)\times (n+1)}$ be an invertible diagonal matrix such that $DP\cS=\Lambda^{p}_+\times\{0\}^{n-p}$. We consider the decomposition pair:
    \[ \vp_d:\R^{p+1}\times \R^{n-p}\to\Rex,\quad \vp_d(y,z):=\iota_{\Lambda^p_+}(y) +\iota_{\{0\}}(z), \quad F(x):=P^{-1}D^{-1}x. \]
    Clearly, we have $\vp=\vp_d\circ F$ and $F(\bar x)=0$ with $F$ being a nonsingular linear mapping. Hence, $\vp$ is $C^2$-fully decomposable at $\bar x$. Moreover, the support set $\pvpd$ for $\vp_d=\sigma_\pvpd$ is given by $-\Lambda^{p}_+\times \R^{n-p}$, which is $C^2$-cone reducible everywhere since $\Lambda^{p}_+$ is $C^2$-cone reducible everywhere, cf. \cite{shapiro2003sensitivity}.
  
    Next, we consider the case where $\bar x\neq 0$. If $\|\bar x_{2:n+1}\| < \bar x_1$, then $\cS$ locally agrees with $\cH$, and it is clear that $\iota_{\cS}$ is $C^2$-fully decomposable at $\bar x$ with its support set, a linear subspace, being $C^2$-cone reducible. Thus, let us assume $\|\bar x_{2:n+1}\| = \bar x_1 > 0$ (since $\bar x\neq 0$).  Without loss of generality, we assume that $A$ has full row rank. We can further assume $m\leq n-1$; otherwise we have $\dim(\cS)\leq \dim(\cH)\leq1$, which was already discussed above. Let $P_1\in \R^{(n+1)\times m}$ and $P_2\in\R^{(n+1)\times(n-m+1)}$ be matrices whose columns form an orthonormal basis of $\ker(A)^{\perp}$ and $\ker(A)$, respectively. Let $Q=P_2P_2^\top $ be the projection matrix onto $\ker(A)$. We consider the following decomposition pair:
    \[ \vp_d:\R^{m} \times \R \to \Rex,\quad \vp_d(y,z)=\iota_{\{0\}}(y) +\iota_{\R_+}(z), \quad F(x)=\begin{bmatrix}
      P_1^\top (x-\bar x)\\
      F_1(Q(x-\bar x))
    \end{bmatrix}, \]
  where $F_1(y)=\bar x_1+y_1-\sqrt{\sum_{i=2}^{n+1}(\bar x_i+y_i)^2}$. We clearly have $\vp=\vp_d\circ F$, $F(\bar x)=0$, the mapping $F$ is real analytic around $\bar x$, and $\vp_d$ is an indicator function of a polyhedral cone $\mathcal K$. The support set $\mathcal K^\circ$ is then also polyhedral and hence $C^2$-cone reducible. Therefore, we only need to verify that $DF(\bar x)$ is surjective. By the definition of $P_1$ and $P_2$, we only need to show $\nabla (F_1\circ Q)(0)=Q\nabla F_1(0)\neq 0$. It holds that
  \[   \nabla F_1(0)^\top =\begin{bmatrix}
     1 &
     -\frac{\bar x_{2:n+1}^\top}{\|\bar x_{2:n+1}\|}
  \end{bmatrix}.           \]
  %
  Let us set $\tilde x=(\bar x_1,-\bar x_2,\dots,-\bar x_{n+1})^\top$. Using $\bar x_1 = \|\bar x_{2:n+1}\|$, we have $Q\nabla F_1(0)=0$ iff $Q\tilde x=0$ or equivalently $\tilde x\in \ker(A)^\perp$. Let us now assume $\tilde x\in \ker(A)^\perp$. By assumption, we have $\ri(\Lambda^{n}_+)\cap \cH\neq \emptyset$ and we can take $y\in \ri(\Lambda^{n}_+)\cap \cH$. Due to $y\in\cS$, it follows $\bar x-y\in \ker(A)$ and we can infer 
  \[ 0 = \iprod{\tilde x}{\bar x-y} = \sum_{i=1}^{n+1}\tilde x_i(\bar x_i - y_i) = \bar x_1^2 - \bar x_1 y_1 - \sum_{i=2}^{n+1} (\bar x_i^2 - \bar x_i y_i) = \sum_{i=2}^{n+1} \bar x_i y_i - \bar x_1 y_1.   \]
  However, this implies $y_1 = \sum_{i=2}^{n+1} \frac{\bar x_iy_i}{\|\bar x_{2:n+1}\|} \leq \|y_{2:n+1}\|$ which contradicts $y\in \ri(\Lambda^n_+)$, i.e., $\|y_{2:n+1}\|<y_1$. 
  %
  Thus, it holds that $\tilde x\notin \ker(A)^\perp$ and $DF(\bar x)$ is surjective. We can conclude that $\vp$ is $C^2$-fully decomposable at $\bar x$ with decomposition pair $(\vp_d,F)$ and the support set is $C^2$-cone reducible everywhere.

\subsubsection{Matrix Intervals}  We consider the matrix interval
    \[  \cS:= [L,U] := \{X\in\bS^{n}:L\preceq X\preceq U\},\quad L,U \in \bS^n.    \]
    As usual, we assume $\cS\neq\emptyset$. We now show that $\vp=\iota_\cS$ is $C^2$-fully decomposable at every $\bar X\in \cS$ with the associated support set being $C^2$-cone reducible everywhere. Let $\alpha\in \R_+$ be sufficiently large such that $L+\alpha I$ is positive definite. By \cite[Theorem 7.6.4]{horn2012matrix}, there then exist a nonsingular matrix $S$ and a diagonal matrix $D$ such that
    \[  L+\alpha I=S^\top S,\quad U+\alpha I=S^\top DS.   \]
    Notice that
    \begin{align*}
       X \in [L,U] & \quad \iff \quad I\preceq S^{-\top}(X+\alpha I)S^{-1}\preceq D \\ & \quad \iff \quad 0\preceq  S^{-\top}(X+\alpha I)S^{-1}-I\preceq D-I.
    \end{align*}
    Due to $\cS\neq \emptyset$, we have $D-I\succeq 0$ and thus, we can select a nonsingular matrix $W$ and $r \in \{0,1,\dots,n\}$ such that $D-I=W^\top E_rW$ where $E_r = \diag({\mathds 1}_r,0_{n-r})$. This yields
    \[ X \in [L,U] \quad \iff \quad 0\preceq  (WS)^{-\top}(X+\alpha I)(WS)^{-1}-W^{-\top}W^{-1}\preceq E_r. \]
    Next, we define $F:\bS^{n}\to\bS^{n}$ as $F(X):=(WS)^{-\top}(X+\alpha I)(WS)^{-1}-W^{-\top}W^{-1}$. Clearly, $F$ is an invertible affine mapping and we have $\cS =\{X\in\bS^n: 0\preceq F(X)\preceq E_r \}$. 
    
    Given $r \in \{0,1,\dots,n\}$, we consider the following decomposition of a symmetric matrix $A \in \bS^{n}$:
    \begin{equation} \label{eq:app-decomp} A = \begin{pmatrix} A_1 & A_2\\ A_2^\top & A_3 \end{pmatrix}, \quad A_1\in\bS^{r}, \; A_2\in\R^{r\times (n-r)}, \; A_3\in\bS^{n-r}. \end{equation}
    Then, we have $0\preceq A \preceq E_r$ if and only if $0\preceq A_1\preceq I_r$, $A_2=0$ and $A_3=0$. Indeed, let $A$ be given with $0 \preceq A \preceq E_r$ and let $x = (0,x_2)^\top\in\R^{n}$ with $x_2\in \R^{n-r}$ be arbitrary. Then, by assumption, it follows $x^\top A x = 0$ which, due to $A \succeq 0$, is further equivalent to $\|Ax\| = 0$. Hence, it holds that $A_2x_2=A_3x_2=0$ for all $x_2\in \R^{n-r}$ which implies $A_2=0$ and $A_3=0$. Moreover, taking $x=(x_1,0)^\top$ with $x_1\in\R^r$ and repeating the same steps, we can infer $0\preceq A_1\preceq I_r$. The implication ``$0\preceq A_1\preceq I_r$, $A_2=0$ and $A_3=0$ $\implies$ $0\preceq A \preceq E_r$'' is immediate. 

    Based on the decomposition \cref{eq:app-decomp}, we now introduce the mapping 
    \[  G:\bS^{n}\to \bS^{r}\times \R^{r\times (n-r)}\times \bS^{n-r},\quad G(A) = (A_1,A_2,A_3). \]
    Clearly, $G$ is an invertible linear mapping and we can equivalently represent the set $\cS$ as follows $\cS=\{X\in\bS^n: G(F(X))\in \mathcal K_1\}$, where
    \[ \mathcal K_1:=\{(A_1,A_2,A_3)\in\bS^{r}\times \R^{r\times (n-r)}\times \bS^{n-r}: A_1\in [0,I_r],\, A_2=0, \,A_3=0\}.   \]
    Let $\bar X \in \mathcal S$ be arbitrary with $G(F(\bar X))=(\bar B,0,0)$ and let $1 \geq \lambda_1(\bar B) \geq \dots \geq \lambda_r(\bar B) \geq 0$ denote the eigenvalues of $\bar B$ in decreasing order. Furthermore, let $1\geq \mu_1>\dots>\mu_q \geq 0$ and $d_1,\dots,d_q$ denote the distinct eigenvalues of $\bar B$ and their multiplicities, i.e., 
    \[\mu_j = \lambda_{c_j+1}(\bar B) = \dots = \lambda_{c_j+d_j}(\bar B), \quad c_j := {\sum}_{i=1}^{j-1} d_i, \quad j = 1,\dots,q, \]
    and set $\mathcal I_j := \{c_j+1,\dots,c_j+d_j\}$, $j = 1,\dots,r$.
    For $B \in \bS^r$, let $P_1(B), P_q(B)\in \bS^{r}$ be the orthogonal projections onto the eigenspaces associated with the collection of eigenvalues $\{\lambda_i(B): i \in \mathcal I_j\}$, $j \in \{1,q\}$, respectively. (Without loss of generality, we assume $1 \neq q$). $P_1$ and $P_q$ are well-defined and real analytic on a neighborhood $\cN$ of $\bar B$, see \cite{shapiro2002differentiability}. Let $E_1\in\R^{r\times d_1}$ and $E_q\in \R^{r\times d_q}$ be  orthonormal bases for the eigenspaces corresponding to the eigenvalues $\{\lambda_i(\bar B): i \in \mathcal I_j\}$, $j \in \{1,q\}$. Then, $P_1(B)E_1$ and $P_q(B)E_q$ have full rank for all $B\in \cN$ by reducing $\cN$ if necessary. Finally, we define $Q_1(B)\in\R^{r\times d_1}$, $Q_q(B)\in \R^{r\times d_q}$ as the matrices obtained by Gram-Schmidt orthogonalization of $P_1(B)E_1$ and $P_q(B)E_q$. The mappings $Q_1$ and $Q_q$ still depend smoothly on $B$ and are orthonormal bases of the eigenspaces corresponding to the eigenvalues $\{\lambda_i(B): i \in \mathcal I_j\}$, $j \in \{1,q\}$. Locally around $G(F(\bar X))$, we can now represent $\mathcal K_1$ via: 
    
    \[ \mathcal K_1=\{(A_1,A_2,A_3)\in \cN \times \R^{r\times (n-r)}\times \bS^{n-r}: H(A_1,A_2,A_3)\in \mathcal K_2\},                  \] 
    where $H: \cN\times \R^{r\times (n-r)}\times \bS^{n-r}\to \bS^{d_1}\times \bS^{d_q} \times \R^{r\times (n-r)}\times \bS^{n-r}$ is given by
   \[ H(A_1,A_2,A_3):=(Q_1(A_1)^\top A_1Q_1(A_1)-\mu_1I,Q_q(A_1)^\top A_1Q_q(A_1)-\mu_qI, A_2,A_3) \]
    and the closed convex cone $\mathcal K_2$ is defined as  
    \[\mathcal K_2 := \begin{cases} \bS^{d_1}\times \bS^{d_q} \times \{0\} \times \{0\} & \text{if } \mu_1,\mu_q \in (0,1), \\  \bS^{d_1}_-\times \bS^{d_q} \times \{0\} \times \{0\}  & \text{if } \mu_1 = 1,\mu_q > 0, \\ \bS^{d_1}\times \bS^{d_q}_+ \times \{0\} \times \{0\}  & \text{if } \mu_1 < 1 ,\mu_q = 0, \\ \bS^{d_1}_-\times \bS^{d_q}_+ \times \{0\} \times \{0\}  & \text{if } \mu_1 = 1,\mu_q = 0.\end{cases} \]
%
    %
This construction closely follows \cite[Examples 3.98 and 3.140]{bonnans2013perturbation}. In particular, following \cite{bonnans2013perturbation}, $H$ is real analytic on a neighborhood of $G(F(\bar X))$, $DH(G(F(\bar X)))$ is surjective, and we have $H(G(F(\bar X))) = 0$. 
Thus, setting $\Xi=H\circ G\circ F$, it holds that $\Xi(\bar X)=0$,
\[ \vp(X)=\iota_\cS(X)=\iota_{\mathcal K_2}(\Xi(X)), \quad \text{for all $X\in\bS^n$ sufficiently close to $\bar X$}, \]
and $D\Xi(\bar X)$ is surjective. Furthermore, based on our earlier discussions and thanks to \cref{remark4-1}, the support set $\mathcal K_2^\circ$ can be easily shown to be $C^2$-cone reducible everywhere. This verifies that the matrix interval $\cS$ and its associated indicator function $\iota_\cS$ are $C^2$-cone reducible and $C^2$-fully $u$-decomposable at $\bar X$, respectively.

\subsection{The Ky Fan \texorpdfstring{$k$}{k}-Norm} Finally, we consider the Ky Fan $k$-norm
     $$
     \|\cdot\|_{(k)}: \mathbb{R}^{m \times n} \rightarrow \mathbb{R}_{+},\quad \vp(X):=\|X\|_{(k)}:=\sum_{i=1}^k \sigma_i(X), \quad k \in\{1, \ldots, m\},
     $$
     which denotes the sum of the $k$-largest singular values. We assume $m\leq n$ for simplicity. In \cite[Example 5.3.18]{milzarek2016numerical}, it is shown  that $\vp$ is $C^2$-fully decomposable at every $\bar X\in \R^{m\times n}$. Here, we verify that the support set $\pvpd$ associated with the respective decomposition pair $(\vp_d \equiv \sigma_\pvpd,F)$ is $C^2$-cone reducible and hence, \cref{thm_pdecom_strict_twice} is applicable to the Ky Fan $k$-norm. We consider the following cases. \\[1mm]
     \noindent\textbf{Case 1:} $\sigma_k(\bar X)>0$. In this case, the outer function $\vp_d$ defined in \cite[Example 5.3.18]{milzarek2016numerical} is given by
     \[  \vp_d:\R\times\bS^{r}\to \R,\quad \vp_d(t,S)=t+{\sum}_{j=1}^{k_0}\lambda_i(S),      \]
     where $k_0\leq r$ are two positive integers and $\lambda_1(S),\dots,\lambda_{k_0}(S)$ denote the $k_0$-largest eigenvalues of $S$. Applying von Neumann's trace inequality, we have $\pvpd=\{1\}\times\cB$ with $\cB=\{B\in \bS^r: 0\preceq B\preceq I, \tr(B)=k_0\}$ and it suffices to show that $\cB$ is $C^2$-cone reducible. We assume $k_0<r$; otherwise $\vp_d(t,S)=t+\tr(S)$ is a linear function and the support set is a singleton. Let us fix $B\in\cB$ and let $\mu_1>\dots>\mu_q$ denote the distinct eigenvalues of $B$ with multiplicities $d_1,\dots,d_q$. Let $P_1(X),\dots,P_q(X)\in \bS^{r\times r}$ denote the projections onto the eigenspaces of $X$. By continuity, the eigenvalues of $X$ stay in distinct boxes around $\mu_1,\dots,\mu_q$ if $X$ is sufficiently close to $B$. Furthermore, the mappings $P_i$, $i = 1,\dots,q$ are well-defined and real analytic on a neighborhood $\cN$ of $B$, \cite{shapiro2002differentiability}.  Let $U_1\in\R^{r\times d_1},\dots,U_q\in \R^{r\times d_q}$ be the orthonormal bases for the eigenspaces $E_1,\dots,E_q$ of $B$ corresponding to the eigenvalues $\mu_1,\dots,\mu_q$. Then, $P_1(X)U_1,\dots,P_q(X)U_q$ have full rank for all $X$ close to $B$. Let $L_1(X)\in\R^{r\times d_1},\dots,L_q(X)\in \R^{r\times d_q}$ be the matrices obtained by performing Gram-Schmidt orthogonalization on $P_1(X)U_1,\dots,P_q(X)U_q$. $L_1(X),\dots,L_q(X)$ are orthonormal bases for the eigenspaces of $X$ corresponding to eigenvalues around $\mu_1,\dots,\mu_q$. We consider the real analytic mapping:
     \[  \Xi:\cN\to \times_{i=1}^{q}\bS^{d_i},\quad \Xi(X):=[L_1(X)^\top XL_1(X)-\mu_1I,\dots,L_q(X)^\top XL_q(X)-\mu_qI].       \]
     By mimicking the proof in \cite[Example 3.140]{bonnans2013perturbation}, it can be shown that $\Xi(B)=0$ and $D\Xi(B)$ is surjective. Next, we define the closed convex cone $\cK \subseteq \times_{i=1}^{q}\bS^{d_i}$ consisting of all tuples $(Q_1,\dots,Q_q)$ satisfying 
     \begin{align*} {\sum}_{i=1}^q\tr(Q_i)=0 \quad \text{and} \quad
        \begin{cases}  Q_1\preceq  0  & \text{if } \mu_1=1  \\ 
        Q_q\succeq 0 & \text{if }  \mu_q=0.
        \end{cases} 
     \end{align*}
     Clearly and by construction, $\cB$ locally agrees with $\{X\in \bS^r: \Xi(X)\in \cK\}$. Hence, $\cB$ is $C^2$-cone reducible at $B$ for all $B\in \cB$. \\[1mm]
     \noindent\textbf{Case 2:} $\sigma_k(\bar X)=0$. Following \cite[Example 5.3.18]{milzarek2016numerical}, the outer function $\vp_d$ is given by
     \[  \vp_d:\R\times\R^{r\times s}\to \R,\quad \vp_d(t,X)=t+\|X\|_{(k_0)},      \]
     where we assume $k_0\leq r \leq s$ without loss of generality. In this case, the support set takes the form $\pvpd=\{1\}\times \cB$, where $\cB=\{B\in \R^{m\times n}: \|B\|_*\leq k_0,~\|B\|_2\leq 1 \}$ and $\|\cdot\|_*$ is the nuclear norm, cf. \cite[Exercise IV.2.12]{bhatia1997matrix}. It suffices to show that $\cB$ is $C^2$-cone reducible everywhere. Fixing $B\in \cB$, let $\mu_1>\dots>\mu_q$ denote the distinct singular values of $B$ with multiplicities $d_1,\dots,d_q$. We consider the following SVD of $B$:
     \[    B=U[\Sigma(B)~~0]V,\quad V=[V_1~~V_2]^\top.          \]
     We now define the linear operator $T:\R^{r\times s}\to \bS^{r+s}$:
     \[   T(X)=\begin{pmatrix}
       0 & X \\
       X^\top & 0
     \end{pmatrix}.    \]
     Given the SVD of $B$, the corresponding SVD of $T(B)$ is:
     \begin{align*}
       T(B)=P(B)\begin{pmatrix}
         \Sigma(B) & 0 & 0 \\
         0 & 0 & 0 \\
         0 & 0 & -\Sigma(B)
       \end{pmatrix}P(B)^\top,\quad P(B)=\frac{1}{\sqrt{2}}\begin{pmatrix}
         U & 0 & U\\
         V_1 & \sqrt{2}V_2 & -V_1
       \end{pmatrix}.
     \end{align*}
      If $\sigma_q(B)>0$, then we define the matrices $L_1(X)\in\R^{r+s\times d_1},\dots,L_q(X)\in \R^{r+s\times d_q}$ whose columns are orthonormal bases of the eigenspaces of $T(X)$ corresponding to all the eigenvalues around $\mu_1,\dots,\mu_q$. As argued in \textbf{Case 1}, the matrices $L_i(X)$ are real analytic on a neighborhood $\cN$ of $B$ for all $i\in [q]$. Consider the analytic mapping: 
      \begin{align*}
        &\Xi: \cN\to \times_{i=1}^{q}\bS^{d_i},\\
        & \Xi(X)=[L_1(X)^\top T(X)L_1(X)-\mu_1I,\dots,L_q(X)^\top T(X)L_q(X)-\mu_qI].
    \end{align*}
     By mimicking the proof in \cite[Proposition 4.3]{ding2012introduction}, $D\Xi(B)$ can be shown to be surjective. We then define the closed convex cone $\cK\subseteq \times_{i=1}^{q}\bS^{d_i}$ via $(Q_1,\dots,Q_q)\in \cK$ iff     
     \begin{align*}
      \begin{cases} Q_1\preceq 0   & \text{if } \mu_1=1, \\ \sum_{i=1}^q\tr(Q_i)\leq 0  & \text{if } \|B\|_*=k_0\end{cases} 
     \end{align*}   
     Hence, locally around $B$, it holds that $\pvpd=\{X\in \R^{r\times s}:\Xi(X)\in \cK\}$, which proves that $\pvpd$ is $C^2$-cone reducible at $B$. Next, we consider the case $\sigma_q(B)=0$. In this case, the number of columns of $L_q(B)$ is larger than $d_q$ since $T(B)$ contains additional zero eigenvalues. Thus, the matrix $L_q(X)^\top XL_q(X)$ may contain eigenvalues that are not singular values of $X$ and our previous construction of $\Xi$ is no longer applicable. However, by mimicking the strategy in \textbf{Case 1}, we can select $L_q(X)\in \R^{r\times d_q}$ and $R_q(X)\in \R^{s\times (d_q+s-r)}$ whose columns are orthonormal bases of the eigenspace of $XX^\top$ and $X^\top X$ corresponding to the eigenvalues around $0$, and therefore $L_q$ and $R_q$ are real analytic on a neighborhood of $B$. In this case, we have $\|X\|_*=\sum_{i=1}^{q-1}\tr(L_i(X)^\top T(X)L_i(X))+\|L_q(X)^\top XR_q(X)\|_*$ for all $X$ in a neighborhood $\cN$ of $B$ and we can redefine the analytic mapping $\Xi$ as follows:
     \begin{align*}
      &\Xi: \cN\to\times_{i=1}^{q-1}\bS^{d_i}\times \R^{d_q\times (d_q+s-r)},\\
      &\Xi(X)=[L_1(X)^\top T(X)L_1(X)-\mu_1I,\dots, L_q(X)^\top XR_q(X)]. 
     \end{align*}
     Similar to \cite[Example 5.3.18]{milzarek2016numerical}, it can be shown that $D\Xi(B)$ is surjective. We define the closed convex cone $\cK\subseteq\times_{i=1}^{q-1}\bS^{d_i}\times \R^{d_q\times (d_q+s-r)}$ via $(Q_1,\dots,Q_q)\in \cK$ via 
     \begin{align*}
      (Q_1,\dots,Q_q)\in \cK \quad \iff \quad \begin{cases} Q_1\preceq 0   & \text{if } \mu_1=1, \\  \sum_{i=1}^{q-1}\tr(Q_i)+\|Q_q\|_*\leq 0  & \text{if } \|B\|_*=k_0\end{cases} 
     \end{align*}  
     Locally around $B$, we then have $\pvpd=\{X\in \R^{r\times s}:\Xi(X)\in \cK\}$, which proves that $\pvpd$ is $C^2$-cone reducible at $B$ (after invoking \cref{remark4-1}).

     We refer to \cite{ding2017kyfan} for additional variational properties of the Ky Fan $k$-norm.

 

\appendix
\section{Proof of Technical and Auxiliary Results}
\subsection{Proof of \texorpdfstring{\cref{prop_equi_quad_diff}}{Proposition 2.2.}} \label{app:sec:equi-epi}

\begin{proof}
        We first verify ``(i)\,$\implies$\,(ii)''. By \cite[Exercises 13.18, 13.35, and 13.45]{rockafellar2009variational}, $\envs$ is twice semidifferentiable at $\bar z$ with $\frac{1}{2} (\envs)^{\prime\prime}(\bar z;\cdot) =e_\tau(\frac{1}{2}\rd^2\vp(\bar x|\bar v))(\cdot)$. Since $\frac{1}{2}\rd^2\vp(\bar x|\bar v)$ is generalized quadratic and $h \mapsto \rd^2\vp(\bar x|\bar v)(h)+\rho\|h\|^2$ is convex by \cite[Proposition 13.49]{rockafellar2009variational}, the mapping  $e_\tau(\frac{1}{2}\rd^2\vp(\bar x|\bar v))$ is a quadratic function; cf. \cref{eq:use-in-app}. Using \cite[Example 13.8, Corollary 13.42]{rockafellar2009variational}, this implies that $\envs$ is twice differentiable at $\bar z$.
        
We continue with ``(ii)\,$\implies$\,(i)''. As before, \cite[Exercise 13.45]{rockafellar2009variational} implies that $\vp$ is properly twice epi-differentiable at $\bar x$ for $\bar v$ and $\frac{1}{2} (\envs)^{\prime\prime}(\bar z;\cdot) =e_\tau(\frac{1}{2}\rd^2\vp(\bar x|\bar v))(\cdot)$. By \cite[Lemma 2.1]{planiden2018epi}, we have
        $e_\tau(\tfrac{1}{2}\rd^2\vp(\bar x|\bar v))(h)=\|h\|^2/\tau-g^*_{\tau}({h}/{\tau})$,                    
        where $g_\tau(h)=\frac{1}{2}\rd^2\vp(\bar x|\bar v)(h)+\frac{1}{2\tau}\|h\|^2$. By assumption, $e_\tau(\frac{1}{2}\rd^2\vp(\bar x|\bar v))$ is a quadratic function and hence, $g^*_\tau$ has to be quadratic as well. Moreover, by \cite[Proposition 13.49]{rockafellar2009variational}, $g_\tau$ is a convex, proper, lsc mapping. Hence, applying \cite[Theorem 13.32]{bauschke2017convex} and \cite[Proposition 4.13]{planiden2018epi}, $g_\tau^{**}=g_\tau$ is generalized quadratic implying that $\rd^2\vp(\bar x|\bar v)$ is generalized quadratic.

Next, we consider ``(ii)\,$\iff$\,(iii)''. By \cite[Proposition 13.37]{rockafellar2009variational}, we have $\nabla \envs =\frac{1}{\tau}(I-\proxs)$ locally around $\bar z$. Hence, this shows that twice differentiability of $\envs$ at $\bar z$ is equivalent to differentiability of $\proxs$ at $\bar z$. Moreover, using $e_\tau(\frac{1}{2}\rd^2\vp(\bar x|\bar v))(\cdot) =\frac{1}{2}(\envs)^{\prime\prime}(\bar z; \cdot)$, twice differentiability of $\envs$ at $\bar z$ already implies that $e_\tau(\frac{1}{2}\rd^2\vp(\bar x|\bar v))$ is a quadratic function. Hence, (ii) and (iii) are equivalent. \end{proof}

\subsection{Proof of \texorpdfstring{\cref{prop2-5}}{Proposition 2.4}} \label{app:sec:prop-technical}

\begin{proof}
    The implication ``(i)\,\!$\implies$\,\!(ii)'' follows from \cite[Theorem 7.37]{rockafellar2009variational}. (The eventual prox-boundedness of $\{f_k\}_k$ is ensured by assumption). Next, let condition (ii) hold. Let $h$ be a proper, lsc, $\rho$-weakly convex function. Then, by \cite[Corollary 3.4 a)]{hoheiselproximal}, we have $\mathrm{env}_{\tau h}(x) = \frac{1}{2\tau} \|x\|^2 - (h + \frac{1}{2\tau}\|\cdot\|^2)^*(x/\tau)$ for all $x\in\Rn$ and $\tau\rho \in (0,1)$. This expression also holds if $\rho = 0$, $\tau > 0$. Let us further define $h_k := (f_k + \frac{1}{2\tau}\|\cdot\|^2)^*$ and $h := (f + \frac{1}{2\tau}\|\cdot\|^2)^*$. By assumption, we can infer $h_k(x) \to h(x)$ for all $x$. Thus, setting $g_k := f_k + \frac{\rho}{2}\|\cdot\|^2$, $g := f + \frac{\rho}{2}\|\cdot\|^2$, and $\lambda := \frac{\tau}{1-\tau\rho}$, it follows
    \begin{align*}
    \mathrm{env}_{\lambda g_k}(x) &= \frac{1}{2\lambda}\|x\|^2 - \left(f_k + \tfrac{\rho}{2}\|\cdot\|^2 + \tfrac{1}{2\lambda}\|\cdot\|^2\right)^*({x}/{\lambda}) = \frac{1}{2\lambda}\|x\|^2 -h_k({x}/{\lambda}) \\ & \hspace{-4ex} \to  \frac{1}{2\lambda}\|x\|^2 -h({x}/{\lambda}) = \frac{1}{2\lambda}\|x\|^2 - \left(f + \tfrac{\rho}{2}\|\cdot\|^2 + \tfrac{1}{2\lambda}\|\cdot\|^2\right)^*({x}/{\lambda})  = \mathrm{env}_{\lambda g}(x). 
    \end{align*}
    Due to the convexity of $g_k$ and $g$, we can now apply \cite[Theorem 7.37 (b${}^\prime$)]{rockafellar2009variational} which implies $g_k\elimit g$. Finally, by \cite[Exercise 7.8 (a)]{rockafellar2009variational}, this yields $f_k = g_k - \frac{\rho}{2}\|\cdot\|^2 \elimit g - \frac{\rho}{2}\|\cdot\|^2 = f$. 
\end{proof}

\vspace{-3ex}
\bibliographystyle{siamplain}
\bibliography{bib_part1}

\end{document}